\documentclass[paper,onefignum,onetabnum]{siamart190516}

\usepackage{lipsum}
\usepackage{amsfonts}
\usepackage{graphicx}
\usepackage{epstopdf}
\usepackage{algorithmic}

\ifpdf
\hypersetup{
  pdftitle={...},
  pdfauthor={...}
}
\fi

\newsiamremark{remark}{Remark}
\crefname{remark}{Remark}{Remark}

\newcommand{\NN}{\mathbb{N}}
\newcommand{\QQ}{\mathbb{Q}}
\newcommand{\RR}{\mathbb{R}}

\newcommand{\cC}{\mathcal{C}}

\newcommand{\cK}{\mathcal{K}}

\newcommand{\cR}{\mathcal{R}}
\newcommand{\cS}{\mathcal{S}}

\newcommand{\id}{\operatorname{id}}

\newcommand{\linspan}{\operatorname{span}}

\newcommand{\rd}{\mathrm{d}}

\newcommand{\re}{\mathrm{e}}

\headers{Compactness of Fixed Point Maps}{G. Dirr}

\title{Compactness of Fixed Point Maps and the Ball-Marsden-Slemrod Conjecture\thanks{Submitted to the editors DATE.}
  }

\author{Gunther Dirr\thanks{Institute of Mathematics, University of W\"{u}rzburg, Emil-Fischer-Str.~40, 97074 W\"{u}rzburg, Germany 
  (\email{dirr@mathematik.uni-wuerzburg.de}).} 
}

\begin{document}

\maketitle

\begin{abstract}
  Given a parameter dependent fixed point equation $x = F(x,u)$, we derive an abstract compactness principle for the fixed
  point map $u \mapsto x^*(u)$ under the assumptions that (i) the fixed point equation can be solved by the contraction
  principle and (ii) the map $u \mapsto F(x,u)$ is compact for fixed $x$.

  This result is applied to infinite-dimensional, semi-linear control systems and their reachable sets. More precisely,
  we extend a non-controllability result of Ball, Marsden, and Slemrod \cite{BMS1982} to semi-linear systems. First we
  consider $L^p$-controls, $p>1$. Subsequently we analyze the case $p=1$.
\end{abstract}

\begin{keywords}
  Compact operators, compactness of fixed point maps, reachable/attainable sets, controllability of bilinear/semi-linear
  systems, Ball-Marsden-Slemrod conjecture
\end{keywords}

\begin{AMS}
  93B05, 93C20, 93C25, 47H10, 47J35
\end{AMS}

\section{Introduction}
\label{sec:intro}

This paper was inspired by a recent result of Boussa\"{i}d, Capo\-nigro, and Chambrion \cite{BCC2019a} (see also
\cite{BCC2019b,BCC2020}) who addressed an ``old'' conjecture of Ball, Marsden, and Slemrod \cite{BMS1982} about the
non-controllability of bilinear control systems with infinite-dimensional state space. For similar
results concerning the linear case we refer to e.g..~\cite{Triggiani1975}.

To be more precise, let us consider a semi-linear
control system on a (possibly complex) Banach space $X$ given by
\begin{equation}
  \label{eq1:BMS}
  \dot{x}(t) = Ax(t) + u(t)f\big(x(t)\big)\,, \quad x(0) = \xi_0 \in X
\end{equation}
where $A: D(A) \to X$ is the (possibly unbounded) infinitesimal generator of a $C^0$-semigroup of bounded linear operators 
while $f:X \to X$ is supposed to be locally Lipschitz\footnote{Actually, in \cite{BMS1982} the authors assume $f$ to be $C^1$
  but their existence result (cf.~Thm.~2.5) works for locally Lipschitz maps as well.} and linearly bounded. Moreover,
the control $t \mapsto u(t)$ is assumed to be in some appropriate $L^p$-space with\footnote{As usual, the notation
  $p \geq 1$ includes the case $p = \infty$.} $p \geq 1$. Obviously, \eqref{eq1:BMS} boils down to a bilinear control system
if $f: X \to X$ is linear.

We call $x: [0,T] \to X$ a \emph{classical solution}\footnote{In particular, if $A$ generates an analytic
    semigroup, the concept of a ``classical solution'' is often weakened in the sense that $x: [0,T] \to X$ is required to
    be continuous on the closed interval $[0,T]$ but continuously differentiable only on the half-open interval $(0,T]$.
    Here, for simplicity, we work with the stronger notion defined above.} of \eqref{eq1:BMS} if $x$ is continuously
differentiable and satisfies \eqref{eq1:BMS} for all $t \in [0,T]$ and \emph{Carath\'{e}odory solution}\footnote{In
    the literature Carath\'{e}odory solutions are also called \emph{strong solutions}, cf.~\cite[Sec.~4.2]{Pazy1983}.
    Moreover, absolute continuity of $x:[0,T] \to X$ is in general not sufficient to guarantee differentiability
      almost everywhere. The implication
      ``absolute continuity $\Longrightarrow$ differentiability almost everywhere with $L^1$-derivative''
      holds if and only if $X$ has the Radon-Nikodym property which is the case for all reflexive and, in particular, all finite
      dimensional Banach spaces, cf.\cite[Thm.~1.2.6 and Cor.1.2.7]{ABHN2011}.}
  if $x$ is continuous, almost everywhere differentiable on $[0,T]$ with $L^1$-derivative and satisfies \eqref{eq1:BMS}
  for almost all $t \in [0,T]$ whereas a \emph{mild solution} of \eqref{eq1:BMS} has to fulfill the integral equation
\begin{equation}
  \label{eq2:BMS}
  x(t) = \re^{tA}\xi_0 + \int_0^t\re^{(t-s)A} u(s)f\big(x(s)\big) \,\rd s
\end{equation}
for all $t \in [0,T]$. Here, $\big(\re^{tA}\big)_{t \geq 0}$ denotes the strongly continuous one-parameter semigroup of linear
operators generated by $A$. It is well known that \eqref{eq2:BMS} has a unique mild solution, denoted by $x(\cdot,\xi_0, u)$,
once $\xi_0 \in X$ and $u \in L^1\big([0,T],\RR\big)$ are fixed, cf.~\cite{BMS1982} or \cite[Sec.~6.1]{Pazy1983}.
For finite-dimensional $X$, the concepts of mild and Carath\'{e}odory solutions coincide,
  cf.~\cite[Chap.~2]{CodLev87} or \cite[App.~C]{Sontag98}, as can be seen by the fundamental theorem of calculus
  \cite[Thm.~7.11 and Thm.~7.20]{Rudin87} and the fact that in finite dimensions every strongly continuous one-parameter
  semigroup is actually norm-continuous and thus given by the exponential series $\re^{At} = \sum_{k=0}^\infty \frac{t^kA^k}{k!}$.
For infinite-dimensional $X$, however, only the implications
\begin{equation*}
  \text{classical} \quad \Longrightarrow \quad \text{Carath\'{e}odory} \quad \Longrightarrow \quad \text{mild}
\end{equation*}
remain true, cf.~\cite[Sec.~6.1]{Pazy1983}. Further references including those to relevant counter-examples are given in
the bibliographical notes of \cite{Pazy1983}.

In the sequel, we completely focus on mild solutions. The \emph{reachable set} $\cR^p(\xi_0)$ of \eqref{eq1:BMS} for
the initial value $\xi_0 \in X$ and $L^p$ controls is then defined as:
\begin{equation*}
  \cR^p(\xi_0) := \bigcup_{T \geq 0} \cR^p_{\leq T}(\xi_0)\,,
\end{equation*}
where  
\begin{equation*}
  \cR^p_{\leq T}(\xi_0) := \big\{x(t,\xi_0,u)  :  t \in [0,T]\,, u \in L^p\big([0,T],\RR\big)\big\}\,.
\end{equation*}
denotes the \emph{reachable set up to time} $T$ and $x(\cdot,\xi_0,u)$ the corresponding unique mild solution.
Finally, \eqref{eq1:BMS} is called \emph{controllable} under $L^p$ controls if $\cR^p(\xi_0) = X$ for all $\xi_0 \in X$
or, equivalently, if for any pair $\xi_0,\eta_0 \in X$ there exist $T \geq 0$ and $u \in L^p\big([0,T],\RR\big)$ such that
$x(T,\xi_0,u) = \eta_0$. 

Ball, Marsden, and Slemrod \cite{BMS1982} showed that a bilinear system\footnote{Except for minor modifications, the
  proof in \cite{BMS1982} should work for Lipschitz continuous $f$ as well.} (i.e.~a system of the above form with $f$
being linear and bounded) is never controllable via $L^p$-controls if $X$ is infinite-dimensional and $p >1$, cf.~\cite[Thm.~3.6]{BMS1982}.
Actually, they proved that for $p >1$ the reachable set $\cR^p(\xi_0)$ of any $\xi_0 \in X$ is a countable
  union of (relatively) compact sets and has therefore no interior points in $X$ if $X$ is an infinite dimensional Banach space.
Moreover, they conjectured that this also holds
for $p=1$, cf.~\cite[Rem.~3.8]{BMS1982}. Recently, Lampart \cite{Lampart2021} derived an even more striking
  non-controllability result for the Schr\"odinger equation on the unbounded domain $\RR^n$ -- but again under the assumption
  $p>1$. So the original Ball, Marsden, and Slemrod conjecture remained open for almost 40 years till Boussa\"{i}d et
  al.~\cite{BCC2019b,BCC2019a,BCC2020} were able to answer it for the bilinear case in the affirmative.
  Thereafter, Chambrion and Thomann \cite{ChTh2019,ChTh2020} derived first results for abstract semi-linear
  systems with applications to non-linear wave and Schr\"odinger equations. In these two series of papers, the authors
  basically exploit -- besides Baire's category theorem -- two\footnote{In \cite{BCC2019b,BCC2020} the Hilbert
    space case allows a different treatment via Radon-measure-valued (impulsive) controls; in \cite{ChTh2020} the authors
    take advantage of additional smoothing properties of the linear part of the equation.} facts:\medskip
\begin{itemize}
\item
  The standard fixed point iteration\footnote{the Dyson series} for solving \eqref{eq2:BMS}
    yields a uniform estimate of the following form
  \begin{equation*}
    \|x_k(t,\xi_0,u) - x_{k-1}(t,\xi_0,u)\| \leq \frac{M^{k+1} \re^{(k+1)wt}\|f\|^k \big(\|u\big|_{[0,t]}\|_1\big)^k}{k!} \|\xi_0\|\,,
  \end{equation*}
  where $x_k(\cdot,\xi_0,u)$ denotes $k$-th iteration and $M >0$, $\omega >0$ are suitable constants,
  cf.~e.g.~\cite[Prop.~11 and Prop.~15]{BCC2019a}.\smallskip
\item
  The ``approximate'' reachable sets 
  given as the sets of all states which can be reached up to time $T$ from $\xi_0$ by the $k$-th iterates
  $x_k(\cdot,\xi_0,u)$ are relatively compact, cf.~\cite[Lemma 12 and proof of Thm.~2]{BCC2019a}.\medskip
\end{itemize}
Analyzing these ideas naturally raises the question whether this result can be derived from a more general principle
which guarantees that relative compactness of the ``approximate''  reachable sets passes over to their limit, i.e.~to
the reachable set. More precisely, we are interested in the following  problem:\\[2mm]
{\it Given a parameter-dependent fixed point equation
\begin{equation}
  \label{eq:fixed_point-eq}
  x = F(x,u)
\end{equation}
such that \eqref{eq:fixed_point-eq} has a unique solution $x^*(u)$ for all $u$. What can be said about the
compactness of the map $u \mapsto x^*(u)$ under the assumption that $u \mapsto F(x,u)$ is compact for fixed $x$?}

\medskip
In Section \ref{sec:CP_for_FPM}, we prove an abstract compactness principle for parameter-dependent fixed point maps,
cf.~Theorem~\ref{thm:main_1}, under the additional assumption that \eqref{eq:fixed_point-eq} can be solved via the
contraction principle (Banach's fixed point theorem). Moreover, we derive some corollaries of Theorem \ref{thm:main_1}
which turn out to be quite useful for applications to ODEs and PDEs. In Subsection \ref{subsec:application_p>1},
we use our previous findings to show that the fixed point map $u \mapsto x(\,\cdot\,,\xi_0,u)$ of \eqref{eq2:BMS} is
compact for $p > 1$. This immediately allows us 
to generalize the non-controllability result of Ball, Marsden, and Slemrod to semi-linear systems. Unfortunately,
for $p = 1$ the fixed point map $u \mapsto x(\,\cdot\,,\xi_0,u)$ fails to be compact. Nevertheless, in Subsection
\ref{subsec:application_p=1} a ``minor'' detour guided by Corollary \ref{cor:main_3a} allows us to prove relative
compactness of the reachable sets
\begin{equation*}
  \cR^{1,r}_{\leq T}(\xi_0) := \{x(t,\xi_0,u)  :  t \in [0,T]\,, \|u\|_1 \leq r \}\,.
\end{equation*}
For this reason, we can finally extend
the non-controllability statement by Ball, Marsden, and Slemrod to semi-linear systems and all $p \geq 1$.
Moreover, we expect that the derived compactness principle will reveal further non-controllability
results for infinite systems  whenever a contraction argument can be applied to ``construct'' solutions.

\medskip
Concluding, we should notice that there are of course results on ``exact'' (local) controllability
  for infinite dimensional systems. Basically they fall into two categories:
  Either (local) controllability is obtained in a space of higher regularity which is often compactly embedded
  in the original one (cf.~\cite{Beauchard2011} and the references therein) or they fit not into the above
  setting like boundary control problems (cf.~\cite{LaTri1991,Zuazua1988}).


\section{Compactness Principle for Fixed Point Maps}
\label{sec:CP_for_FPM}

\subsection*{Notation and Preliminaries}
\label{subsec:Notation and Preliminaries}

Let $X$ be an arbitrary metric space. Its metric is usually denoted by $d$; its open and closed balls of radius
$r \geq 0$ and center $x$ are referred to as $B_r(x)$ and $K_r(x)$, respectively. Another metric $d'$ on $X$ is
called \emph{strongly equivalent} to $d$ if there exist constants  $M > 0$ and $M' > 0$
such that the estimates 
\begin{equation}
  \label{eq:equivalent_metrics}
  d(x,y) \leq M' d'(x,y) \quad\text{and}\quad d'(x,y) \leq M d(x,y)
\end{equation}
are satisfied for all $x,y \in X$.

Whenever in the following products of metric (topological) spaces occur, they will be equipped with the product
topology. Obviously, if these products are finite the product topology coincides with the topology induced
by the metric
\begin{equation}
  \label{eq:product-metric_inf}
  d_\infty(x,y) := \max_{i = 1, \dots,n }d_i(x_i,y_i)
\end{equation}
or, alternatively, by
\begin{equation*}
  d_1(x,y) := \sum_{i=1}^nd_i(x_i,y_i)\,.
\end{equation*}

\begin{remark}
  \label{rem:equivalent_metrics}
  Note that the above concept of equivalence of metrics is rather ``strong'' in the sense that two metrics $d$ and $d'$ can
  generate the same topology without being \emph{strongly equivalent}. However, it guarantees that the corresponding uniform
  structures coincide and thus completeness is preserved when passing from one metric to another strongly equivalent one.
  Moreover, if $d$ and $d'$ are induced by norms then the above concept boils down to the standard notion of equivalence
  of norms.
\end{remark}

Next, let $X$, $Z$ and $P$ be metric spaces and let $F:X \times P \to Z$, $G:X \to Z$ and $H:P \to Z$ be arbitrary maps.
By the usual abuse of notation\footnote{Of course; if $X$ and $P$ are not disjoint then $F_x$ is not well-defined for
  $x \in X \cap P$. Nevertheless, for simplicity we stick to this common notation and write $F(\cdot,x)$ or $F(x,\cdot)$ instead
  of $F_x$ whenever confusion can occur.}, let $F_x$ and $F_u$ with $x \in X$ and $u \in P$ denote the partial maps
\begin{equation*}
  F_x:P \to Z\,, \quad u \mapsto F_x(u) := F(x,u)
\end{equation*}
and
\begin{equation*}
  F_u:X \to Z\,, \quad x \mapsto F_u(x) := F(x,u)\,,
\end{equation*}
respectively. Then $G:X \to Z$ is termed\medskip
\begin{itemize}
\item
  \emph{(globally) Lipschitz} if there exists a constant $L \geq 0$ such that
  \begin{equation}
    \label{eq:globally_lipschitz}
    d(G(x),G(y)) \leq L d(x,y)
  \end{equation}
  holds for all $x,y \in X$.\smallskip
\item
  \emph{locally Lipschitz} if for every $x_0 \in X$ there exists neighborhood $U_0 \subset X$ of $x_0 \in X$ and a constant
  $L \geq 0$ such that \eqref{eq:globally_lipschitz}
holds for all $x,y \in U_0$.\smallskip
\item
  \emph{Lipschitz on bounded sets} if for every bounded subset $B \subset X$ there exists a constant $L \geq 0$ such
  that \eqref{eq:globally_lipschitz}
  holds for all $x,y \in B$.
\end{itemize}

\begin{remark}
  While an arbitrary map $G:X \to Z$ between metric spaces which is Lipschitz on bounded sets
  is also locally Lipschitz, the converse is in general false. However, if $X$ has the Heine-Borel property
  (i.e.~if ``closed and bounded $=$ compact'') then both notions are equivalent.
  \end{remark}

\noindent
Moreover, the map $F:X \times P \to Z$ is said to be\medskip
\begin{itemize}
\item
  \emph{Lipschitz in $x$} if for every $u \in P$ there exists a constant $L \geq 0$ such that
\begin{equation}
  \label{eq:lipschitz_in_x}
  d(F_u(x),F_u(y)) \leq L d(x,y)
\end{equation}
holds for all $x,y \in X$.\smallskip
\item
  \emph{Lipschitz in $x$ uniformly in $u$} if there exists a constant $L \geq 0$ such that \eqref{eq:lipschitz_in_x}
holds for all $x,y \in X$ and all $u \in P$.\smallskip
\item
  \emph{Lipschitz in $x$ uniformly on bounded sets of $P$} if for every bounded set $B \subset P$ there
  exists a constant $L \geq 0$ such that \eqref{eq:lipschitz_in_x}
holds for all $x,y \in X$ and all $u \in B$.\smallskip
\item
  \emph{Lipschitz in $x$ locally uniformly in $u$} if for every $u_0 \in P$ there exists a neighborhood
  $U_0$ of $u_0 \in P$ and a constant $L \geq 0$ such that \eqref{eq:lipschitz_in_x}
holds for all $x,y \in X$ and all $u \in U_0$.
\end{itemize}

\medskip
\noindent
If in the above definitions the term ``Lipschitz'' is replaced by \emph{locally Lipschitz} then for each $x_0 \in X$
there exists a neighborhood $U_0 \subset X$ of $x_0 \in X$ such that the above conditions are satisfied on $U_0$ instead
of $X$. If ``Lipschitz'' is replaced by \emph{contractive} then the constant $L \geq 0$ can be chosen to be less
than $1$. For instance, $F$ is called \emph{contractive in $x$ uniformly on bounded sets of $P$} if for every
bounded set $B \subset P$ there exists a constant $0 \leq C < 1$ such that
\begin{equation*}
  d(F_u(x),F_u(y)) \leq C d(x,y)
\end{equation*}
is satisfied for all $x,y \in X$ and all $u \in B$.

\medskip

For $Z = X$, we say that $G:X \to X$ is an \emph{eventual contraction} if there exists a natural number
$N \in \NN$ such that $G^N:X \to X$ is a contraction. Consequently, $F:X \times P \to X$ is termed
\emph{eventual contraction in $x$ uniformly on bounded sets of $P$} if for every bounded set $B \subset P$ there
exists a natural number $N \in \NN$ and a constant $0 \leq C < 1$ such that
\begin{equation*}
  d(F^N_u(x),F^N_u(y)) \leq C d(x,y)
\end{equation*}
holds for all $x,y \in X$ and all $u \in B$. Note that the constant $C$ and the power $N \in \NN$ may depend on $B \subset P$.

\medskip

Finally, $H:P \to Z$ is called \emph{compact} if $H$ maps bounded sets of $P$ to relatively compact sets of $Z$. Recall
that in a complete metric space $Z$ a set $S \subset Z$ is \emph{relatively compact}  (i.e.~has compact closure) if and
only if it features the $\varepsilon$-\emph{net property}. This means that for all $\varepsilon > 0$ there exists a finite
$\varepsilon$-\emph{net}, i.e.~there exists a natural number $N \in \NN$ and finitely many $z_i \in S$, $i=1, \dots, N$
such that $S$ is covered by the union of the balls $B_{\varepsilon}(z_i)$, $i=1, \dots, N$. The $\varepsilon$-net property is
also called \emph{total boundedness} or \emph{precompactness}, cf.~\cite[Prelim.]{MeiseVogt97} and \cite[Thm.~A.4]{Rudin91}.
Whenever the involved metric spaces are complete we will use these terms interchangeably. Of course, in any other
case one has to distinguish clearly between relative compactness and total boundedness (precompactness).  

\medskip
\noindent
\textbf{Warning.} Here and henceforth, boundedness of sets is always meant in the metric sense, i.e.~$S \subset P$
is bounded if there exists $r \geq 0$ such that $d(x,y) \leq r$ for all $x,y \in S$.

\begin{remark}
  \label{rem:montel-map}
  \begin{enumerate}
  \item[(a)]
     Compact linear maps between normed vector spaces are always bounded and therefore continuous. However, non-linear maps
     which are compact are not necessarily continuous. Moreover, for linear maps between normed vector spaces the following
     properties are equivalent: (i) bounded sets are mapped to relatively compact ones. (ii) there exits a bounded neighborhood
     of the origin whose image is relatively compact. This equivalence clearly fails for non-linear maps.\smallskip
   \item[(b)]
     In metric spaces which carry an additional linear structure one has to be careful because: (i) boundedness in the
     metric sense is in general not equivalent to boundedness with respect to the underlying vector space topology\footnote{In
       a topological vector space $T$ a subset $B \subset T$ is called (topologically) bounded if for every neighborhood
       $U_0$ of the origin there exits a scalar $r > 0$ such that $B \subset r U_0$ holds. If the topology arises from
       a single norm then both concepts of boundedness coincide. If not -- for instance in ``proper'' Fr\'{e}chet spaces
       -- the two notions differ significantly.}.
     (ii) An equivalence similar to that mentioned in part (a) does not hold any longer as bounded neighborhoods of the
     origin may not exist. This discrepancy led even in the linear case to different notions of compact maps in Fr\'{e}chet
     spaces, cf.~\cite{BonLin94}. 
     %
     %
   \end{enumerate}
\end{remark}

\noindent
After these preliminaries, we can state the main result of this section.

\begin{theorem}[Compactness Principle]
  \label{thm:main_1}
  Let $X$ and $P$ be complete metric spaces and let $F:X \times P \to X$ satisfy the following conditions:\smallskip
  \begin{enumerate}
  \item[(a)]
    $F: X \times P \to X$ is a contraction in $x$ uniformly on bounded sets of $P$.\smallskip
  \item[(b)]
    $F_x:P \to X$ is continuous for all $x \in X$.\smallskip
  \item[(c)]
    $F_x:P \to X$ is compact for all $x \in X$.\smallskip
  \end{enumerate}
  Then the well-defined fixed point map $\Phi: P \to X$  which assigns to each $u \in P$ the unique fixed
  point of $F_u(\cdot)$, i.e.~$\Phi(u) = F\big(\Phi(u),u\big)$, is continuous and compact. If additionally\smallskip
  \begin{enumerate}
  \item[(d)]
    $F:X \times P \to X$ is locally Lipschitz in $u$ locally uniformly in $x$\smallskip
  \end{enumerate}
  then $\Phi: P \to X$ is also locally Lipschitz.
\end{theorem}

\noindent
Before losing ourselves and the reader in the $\varepsilon$-details of the proof of Theorem \ref{thm:main_1}
  it is worth to sketch its road map: Lemma \ref{lem0:continuity} and the (Lipschitz) continuity of the fixed point map
  are standard results, cf.~\cite[Prop.~1.2]{Zeidler86}. In contrast, Lemma \ref{lem1:precompact} and Corollary \ref{cor:precompact}
  comprise the essential ingredients for proving compactness of the fixed point map as they guarantee that
  in each iteration of the map $F_u$ the set of iterates $\{F^n_u(x_0) \;|\; u \in U\}$ is relatively compact if $U \subset P$
  is bounded, where $x_0$ is an arbitrary, but fixed initial point. Thus, in the final proof of Theorem \ref{thm:main_1} we only
  have to ``close the gap'' between $\{F^n_u(x_0) \;|\; u \in U\}$ and $\{F^\infty_u(x_0) \;|\; u \in U\}$ which is obtained
  via the uniform estimate \eqref{eq:unifrom}.

\begin{lemma}
  \label{lem0:continuity}
  Let $F: X \times P \to X$ be locally Lipschitz in $x$ locally uniformly in $u$. \smallskip
  \begin{enumerate}
  \item[(a)]
    If $F$ satisfies property (b) of Theorem \ref{thm:main_1} then it is continuous on $X \times P$.\smallskip
  \item[(b)]
    If $F$ satisfies property (d) of Theorem \ref{thm:main_1} then it is locally Lipschitz on $X \times P$.
  \end{enumerate}
\end{lemma}

\noindent
The straightforward proof is left to the reader.

\begin{lemma}
  \label{lem1:precompact}
  Let $X$ and $P$ be complete metric spaces. Moreover, let $K \subset X$ be relatively compact and $B \subset P$ be
  bounded. If $F: X \times P \to X$ is Lipschitz in $x$ uniformly on bounded sets of $P$ and satisfies property (c)
  of Theorem \ref{thm:main_1} then $F(K \times B)$ is also relatively compact.
\end{lemma}

\begin{proof}
  Let $B \subset P$ be bounded and $\varepsilon > 0$. Then, by assumption, there exists a constant $L > 0$ such that
  \begin{equation}
    \label{eq:lip-1}
    d(F_u(x),F_u(y)) \leq L d(x,y)
  \end{equation}
  holds for all $x,y \in X$ and all $u \in B$. Moreover, as $K$ is relatively compact we can choose a finite
  $\frac{\varepsilon}{2L}$-net $x_i$, $i = 1, \dots, N$ of $K$. Then due to condition (c) the images
  $F(x_i, B)$, $i = 1, \dots, N$ and thus their union are relatively compact. Hence one can find a finite
  $\frac{\varepsilon}{2}$-net $y_j$, $j = 1, \dots, M$ of $\bigcup_{i=1}^N F(x_i,B)$. Now for arbitrary $x \in K$
  and $u \in B$ one can choose $x_i$ and $y_j$ such that $x \in B_{\varepsilon/2L}(x_i)$ and $F(x_i,u) \in B_{\varepsilon/2}(y_j)$.
  Consequently, \eqref{eq:lip-1} implies
  \begin{equation*}
      d\big(F(x,u),y_j\big)  \leq d\big(F(x,u),F(x_i,u)\big) + d\big(F(x_i,u),y_j\big)
      < L d(x,x_i) + \frac{\varepsilon}{2} < \varepsilon 
  \end{equation*}
  i.e.~$y_1, \dots, y_M$ yields a finite $\varepsilon$-net of $F(K \times B)$. Hence $F(K \times B)$ is totally
  bounded and thus relatively compact.
\end{proof}

\begin{corollary}
  \label{cor:precompact}
  Let $X$ and $P$ be complete metric spaces and let $B \subset P$ be bounded.
  Moreover, for $x_0 \in X$ define
  \begin{equation*}
    W_0(x_0) := \{x_0\} \quad \text{and} \quad W_{n+1}(x_0) := F(W_n \times B)\,.
  \end{equation*}
  If $F: X \times P \to X$ is Lipschitz in $x$ uniformly on bounded sets of $P$ and satisfies property (c) of Theorem
  \ref{thm:main_1}, then $W_n(x_0)$ is relatively compact for all $n \in \NN_0$.
\end{corollary}

\begin{proof}
  Follows immediately from Lemma \ref{lem1:precompact}.
\end{proof}

\begin{proof}[Proof of Theorem \ref{thm:main_1}]
  It is well known, cf.~e.g.~\cite[Prop.~1.2]{Zeidler86}, that under conditions (a) and (b) the fixed point map
  $u \mapsto \Phi(u)$ is well-defined (due to the contraction principle) and continuous. For the sake of completeness
  we present the simple arguments in the following:
  
  \medskip
  \noindent
  Let $u \in P$ be given and let $B := B_1(u)$ be a bounded neighborhood of $u$. Then by assumption there
  exists a uniform contraction rate $0 \leq C < 1$ for all $v \in B$. Hence one has
  \begin{equation*}
    \begin{split}
		d\big(\Phi(v),\Phi(u)\big) &= d\big(F(\Phi(v),v),F(\Phi(u),u)\big)\\
      &  \leq d\big(F(\Phi(v),v),F(\Phi(u),v)\big) + d\big(F(\Phi(u),v),F(\Phi(u),u)\big)\\
      &  \leq C d\big(\Phi(v),\Phi(u)\big) + d\big(F(\Phi(u),v),F(\Phi(u),u)\big)\\
    \end{split}
  \end{equation*}
  and therefore
  \begin{equation*}
    d\big(\Phi(v),\Phi(u)\big) \leq \frac{1}{1-C} d\big(F(\Phi(u),v),F(\Phi(u),u)\big)
  \end{equation*}
  for all $v \in B$. Hence the continuity of the map $F_{\Phi(u)}$ shows that $\Phi$ is continuous at $u \in P$.
  If $F$ satisfies additionally condition (d) one finds a constant $L \geq 0$ and neighborhoods $U \subset X$ of
  $\Phi(u)$ and $V \subset P$ of $u \in P$ such that
  \begin{equation*}
    d\big(F(x,v),F(x,w)\big) \leq L d\big(v,w\big)
  \end{equation*}
  for all $x \in U$ and $v,w \in V$. Then the continuity of $\Phi$ at $u \in P$ implies that
  $W := V \cap \Phi^{-1}(U) \cap B$ is again a neighborhood of $u \in P$. Thus we conclude 
  \begin{equation*}
      d\big(\Phi(v),\Phi(w)\big) \leq \frac{1}{1-C} d\big(F(\Phi(v),v),F(\Phi(v),w)\big)
      \leq \frac{L}{1-C} d(v,w)
  \end{equation*}
  for all $v,w \in W$, i.e.~$\Phi$ is locally Lipschitz.

  \medskip
  \noindent
  Finally, let us show that $\Phi$ is also compact. To this end, let $B \subset P$ be bounded and $\varepsilon > 0$.
  From the standard proof of the contraction principle, cf.~e.g.~\cite[Chap.~1]{Zeidler86}, we know that $\Phi(u)$ is
  given by
  \begin{equation*}
    \Phi(u) = \lim_{n \to \infty}F^n_u(x_0)\,,
  \end{equation*}
  where $x_0 \in X$ can be an arbitrary initial point. Moreover, again from the contraction principle, 
  one has the uniform estimate
  \begin{equation}
    \label{eq:unifrom}
    d\big(F^n_u(x_0),F_u^m(x_0)\big) \leq \sum_{k=m}^{n-1} C^k d\big(F_u(x_0),x_0\big)
    \leq \frac{C^m}{1-C} d\big(F_u(x_0),x_0\big)
  \end{equation}
  for all $u \in B$ and all $n > m$. Because $F(x_0,B) \subset X$ is relatively compact by assumption, there exists a
  constant $K \geq 0$ such that $d\big(F_u(x_0),x_0\big) \leq K$ for all $u \in B$. This implies that there
  exists $N \in \NN$ such that 
  \begin{equation*}
    d\big(\Phi(u),F_u^n(x_0)\big) \leq \frac{K C^n}{1-C} < \frac{\varepsilon}{2}
  \end{equation*}
  for all $u \in B$ and all $n \geq N$. Furthermore, by Corollary \ref{cor:precompact} we conclude that
  $\{F_u^N(x_0)  :  u \in B\} \subset W^N(x_0)$ is relatively compact. Hence there exists a finite $\frac{\varepsilon}{2}$-net
  $y_1, \dots, y_M$ of $\{F_u^N(x_0)  :  u \in B\}$ and thus for all $u \in B$ we can find $y_j$ such that
  $d(F^N_u(x_0),y_j) < \frac{\varepsilon}{2}$. This yields the estimate
  \begin{equation*}
    d\big(\Phi(u),y_j\big) \leq d\big(\Phi(u),F^N_u(x_0)\big) + d\big(F^N_u(x_0),y_j\big) < \varepsilon
  \end{equation*}
  i.e.~$y_1, \dots, y_M$ yields a finite $\varepsilon$-net of $\Phi(B)$. This completes the proof.
\end{proof}

\begin{remark}
  \label{rem:alternative_proof}
  Actually, there is a beautiful alternative to prove the above theorem, cf.~\cite{Glueck}: Fix a bounded subset $B \subset P$
  and consider the function space $C_{\rm tb}(B,X) \subset C(B,X)$ of all continuous maps with totally bounded
  image. Then, under the conditions of Theorem \ref{thm:main_1}, one can show that the map $T: \Psi \mapsto T(\Psi)$ with
  $T(\Psi)(u) := F\big(\Psi(u),u\big)$ defines a contraction from $C_{\rm tb}(B,X)$ into it self. Since $C_{\rm tb}(B,X)$ is
  closed the map $T$ has a unique fixed point $\Psi_* \in C_{\rm tb}(B,X)$ which which satisfies
  $F\big(\Psi_*(u),u\big) = \Psi_*(u)$ for all $u \in B$. Hence $\Psi_*$ coincides with the restriction $\Phi\big|_B$ of
  the fixed point map. Thus the image of $\Phi$ restricted to $B$ is totally bounded ($=$ relatively compact).
\end{remark}

\medskip
\noindent
The following two corollaries of Theorem \ref{thm:main_1} are particularly useful for applications in the area of
ODEs and PDEs where the corresponding integral operators sometimes reveal their contractive behavior only after several
iterations or, alternatively, after passing to a strongly equivalent metric, cf.~Subsection \ref{subsec:application_p=1}
and \ref{subsec:application_p>1}, respectively.

\begin{corollary}
  \label{cor:main_1}
  Let $X$ and $P$ be complete metric spaces and let $F:X \times P \to X$ satisfy the following conditions:\smallskip
  \begin{enumerate}
  \item[(a)]
    $F: X \times P \to X$ is Lipschitz in $x$ uniformly on bounded sets of $P$.\smallskip
  \item[(b)]
    $F: X \times P \to X$ is an eventual contraction in $x$ uniformly on bounded sets of $P$.\smallskip
  \item[(c)]
    $F_x:P \to X$ is continuous for all $x \in X$.\smallskip
  \item[(d)]
    $F_x:P \to X$ is compact for all $x \in X$.\smallskip
  \end{enumerate}
  Then the well-defined fixed point map $\Phi: P \to X$  which assigns to each $u \in P$ the unique fixed
  point of $F_u(\cdot)$, i.e.~$\Phi(u) = F\big(\Phi(u),u\big)$, is continuous and compact. If additionally\smallskip
  \begin{enumerate}
  \item[(e)]
    $F:X \times P \to X$ is locally Lipschitz in $u$ locally uniformly in $x$\smallskip
  \end{enumerate}
  then $\Phi: P \to X$ is also locally Lipschitz in $u$.
\end{corollary}

\begin{corollary}
  \label{cor:main_2}
  Let $X$ and $P$ be complete metric spaces and let $F:X \times P \to X$ satisfy the following conditions:\smallskip
  \begin{enumerate}
  \item[(a)]
    For every bounded set $B \subset P$ there exists a strongly equivalent metric $d_B$ on $X$ and a constant
    $0 \leq C < 1$ such that 
    \begin{equation*}
      d_B(F_u(x),F_u(y)\big) \leq  C d_B(x,y)
    \end{equation*}
    for all $x,y \in X$ and all $u \in B$.\smallskip
  \item[(b)]
    $F_x:P \to X$ is continuous for all $x \in X$.\smallskip
  \item[(c)]
    $F_x:P \to X$ is compact for all $x \in X$.\smallskip
  \end{enumerate}
  Then the well-defined fixed point map $\Phi: P \to X$  which assigns to each $u \in P$ the unique fixed
  point of $F_u(\cdot)$, i.e.~$\Phi(u) = F\big(\Phi(u),u\big)$, is continuous and compact. If additionally\smallskip
  \begin{enumerate}
  \item[(d)]
    $F:X \times P \to X$ is locally Lipschitz in $u$ locally uniformly in $x$\smallskip
  \end{enumerate}
  then $\Phi: P \to X$ is also locally Lipschitz in $u$.
\end{corollary}

\noindent
The following lemma serves to reduce the proof of Corollary \ref{cor:main_1} to that of Corollary \ref{cor:main_2}.
An almost identical construction can be found in \cite {Brown18}, Chap.~3. Nevertheless, for the sake
  of self-containedness, we present the short proof.

\begin{lemma}
  \label{lem:equivalent-metric}
  Let $X$ be a metric space and let $F:X\to X$ be given. Then the following two assertions are equivalent:\smallskip
  \begin{enumerate}
  \item[(a)]
    $F$ is Lipschitz and an eventual contraction.\smallskip
  \item[(b)]
    There exists a strongly equivalent metric $d'$ on $X$ such that $F$ is a contraction with respect to $d'$.
  \end{enumerate}
\end{lemma}

\begin{proof}
  (a) $\Longrightarrow$ (b): Assume that $F$ is an eventual contraction, i.e.~there exists $N \in \NN$ and $0 \leq C < 1$
  such that
  \begin{equation}
    \label{eq:hidden-contraction}
    d\big(F^N(x),F^N(y)\big) \leq C d(x,y)
  \end{equation}
  for all $x,y \in X$. This allows us to define $d'$ as follows
  \begin{equation*}
    d'(x,y) := \sup_{n \in \NN_0}\frac{d\big(F^n(x),F^n(y)\big)}{C^{\frac{n}{N}}} < \infty\,.
  \end{equation*}
  Obviously, $d'$ yields a metric which satisfies the estimate $d(x,y) \leq d'(x,y)$ and 
  \begin{equation*}
    \begin{split}
       d'\big(F(x),F(y)\big) &= \sup_{n \in \NN_0}\frac{d\big(F^{n+1}(x),F^{n+1}(y)\big)}{C^{\frac{n}{N}}}\\
      &  = C^{\frac{1}{N}}\sup_{n \in \NN_0}\frac{d\big(F^{n+1}(x),F^{n+1}(y)\big)}{C^{\frac{n+1}{N}}}
      \leq C^{\frac{1}{N}}d'(x,y)\,,
    \end{split}
  \end{equation*}
  i.e.~$F$ is a contraction with respect to $d'$. In order to see the strong equivalence of $d$ and $d'$ one
  exploits \eqref{eq:hidden-contraction} to obtain
  \begin{equation*}
      d'(x,y) = \sup_{n \in \NN_0}\frac{d\big(F^n(x),F^n(y)\big)}{C^{\frac{n}{N}}}
      = \max_{0 \leq n \leq N-1}\frac{d\big(F^n(x),F^n(y)\big)}{C^{\frac{n}{N}}}
  \end{equation*}
  and so Lipschitz continuity of $F$ implies $d'(x,y) \leq M d(x,y)$ with
  \begin{equation*}
    M := \max_{0 \leq n \leq N-1}\frac{L^n}{C^{\frac{n}{N}}}\,,
  \end{equation*}
  where $L \geq 0$ is any Lipschitz constant for $F$.
    
  \medskip
  \noindent
  (b) $\Longleftarrow$ (a): Assume that there exists an strongly equivalent metric $d'$ such that $F$ is a
  contraction with contraction rate $0 \leq C' < 1$ with respect of $d'$. Hence
  \begin{equation*}
      d\big(F(x),F(y)\big) \leq M' d'\big(F(x),F(y)\big) \leq M' C' d'\big(x,y) \leq M M' C' d\big(x,y) \,,
  \end{equation*}
  where $M$ and $M'$ are constants which satisfy \eqref{eq:equivalent_metrics}. This shows that $F$ is
  Lipschitz continuous with $L := MM'C'$. Moreover, choosing $N \in \NN$ such that the following
  estimate holds
  \begin{equation*}
    C := M M'(C')^N < 1
  \end{equation*}
  one obtains 
  \begin{equation*}
    \begin{split}
      d\big(F^N(x),F^N(y)\big) &\leq M' d'\big(F^N(x),F^N(y)\big) \leq M' (C')^N d'\big(x,y) \\
      & \leq M M' (C')^N d\big(x,y) = C d\big(x,y)\,,
    \end{split}
  \end{equation*}
  for all $x,y\in X$, i.e.~$F$ is an eventual contraction.
\end{proof}

\begin{proof}[Proof of Corollary \ref{cor:main_2}]
  First, we briefly demonstrate that the fixed point map $\Phi: P \to X$ is well-defined. To this end, choose
  an arbitrary $u \in P$. Since $\{u\}$ is obviously a bounded subset of $P$ there, by assumption, exists a strongly
  equivalent metric $d_u$ such that $F_u:X \to X$ is a contraction. Thus the contraction principle implies that
  $F_u$ has a unique fixed point $\Phi(u) \in X$. Note that the strong equivalence of $d$ and $d_u$ guarantees
  that $X$ is complete with respect $d_u$.
  
  \medskip
  \noindent
  Next, let us prove (local Lipschitz) continuity of $u \mapsto \Phi(u)$. Therefore, let $u \in P$ and choose a
  bounded and closed\footnote{Later $B$ will play the role of $P$ (cf.~Theorem \ref{thm:main_1}) and therefore we
    have to guarantee that $B$ constitutes a complete metric space.} neighborhood of $u \in P$, e.g.~$B := K_1(u)$.
  Since (local Lipschitz) continuity is preserved when passing to a strongly equivalent metric it suffices to show
  the desired result for the metric $d_{B}$. Now, the restriction $F\big|_{X \times B}$ satisfies readily the assumptions
  of Theorem \ref{thm:main_1} (with respect to the metric $d_{B}$). Thus we can conclude that $\Phi$ is continuous
  at $u \in P$, and---under assumption (d)---that $\Phi$ is even locally Lipschitz.

  \medskip
  \noindent
  Finally, let $B \subset P$ be an arbitrary bounded subset of $P$. We have to show that $\Phi(B)$ is relatively compact
  in $X$. Due to assumption (b) we can choose a strongly equivalent metric $d_{B}$ on $X$. Again it suffices to prove
  that $\Phi(B)$ is relatively compact with respect to $d_{B}$. Hence, by the same argument as above, we may consider
  the restriction $F\big|_{X \times B}$ instead of $F$, and infer from Theorem \ref{thm:main_1} that $\Phi(B)$
  is relatively compact.
\end{proof}

\begin{remark}
  \label{rem:continuiuty}
  One might conjecture that assumption (a) in Corollary \ref{cor:main_1} can be dropped as assumption (b) already
  guarantees a unique fixed point of $F_u(\cdot)$. But that does not work in general, because (b) does not ensure
  continuity of the fixed point map as the following example shows: Let $X=P=\RR$ and define $F: X \times P \to X$ by
  \begin{equation*}
    F(x,u) :=
    \begin{cases}
      u & \text{for $x \in \QQ$,}\\
      -u & \text{for $x \in \RR \setminus \QQ$.}
    \end{cases}
  \end{equation*}
  Then $F_x = F(x,\cdot)$ is continuous and $F$ is an eventual contraction in $x$ uniformly in $u$ as one has
  \begin{equation*}
    F_u^2(x) = F\big(F(x,u),u\big) =
    \begin{cases}
      u & \text{for $u \in \QQ$}\\
      -u & \text{for $u \in \RR \setminus \QQ$.}
    \end{cases}
  \end{equation*}
  Hence
  \begin{equation*}
    \Phi(u) =
    \begin{cases}
      u & \text{for $u \in \QQ$,}\\
      -u & \text{for $u \in \RR \setminus \QQ$,}
    \end{cases}
  \end{equation*}
  is obviously not continuous.
\end{remark}

\noindent
The final result of this section handles a scenario in which one is interested in the compactness of some composition
$G \circ \Phi: P \to Z$ instead of the compactness of the fixed point map $\Phi:P \to X$ itself. The Ball-Marsden-Slemrod
conjecture for $p=1$ falls in this regime and will be tackled with the help of the following corollary in Subsection
\ref{subsec:application_p=1}.

\begin{corollary}
  \label{cor:main_3a}
  Let $X$, $Z$ and $P$ be complete metric spaces and let $F:X \times P \to X$ and $G:X \to Z$ satisfy the
  following conditions:\smallskip
  \begin{enumerate}
  \item[(a)]
    $F: X \times P \to X$ is a contraction in $x$ uniformly on bounded sets of $P$.\smallskip
  \item[(b)]
    $F_x:P \to X$ is continuous for all $x \in X$.\smallskip
  \item[(c)]
    $G:X \to Z$ is Lipschitz continuous.\smallskip
  \item[(d)]
    For $S \subset X$ and $B \subset P$ the following implication holds:
    \begin{equation*}
      \hspace{-12mm}\text{$G(S)$ relatively compact and $B$ bounded}
      \; \Longrightarrow \;
      \text{$G\big(F(S \times B)\big)$ relatively compact}
  \end{equation*}
  \end{enumerate}
  Then the well-defined fixed point map $\Phi: P \to X$ which assigns to each $u \in P$ the unique fixed
  point of $F_u(\cdot)$, i.e.~$\Phi(u) = F\big(\Phi(u),u\big)$, and the composition $G \circ \Phi$ are
  continuous. Moreover,  $G \circ \Phi$ is compact.
\end{corollary}

\begin{proof}
  Obviously it suffices to focus on the compactness of $\Phi$ as continuity follows directly from Theorem \ref{thm:main_1}
  and thus continuity of $G\circ\Phi$ is evident.

  \medskip
  \noindent
  To this end we can proceed as in the proof of Theorem \ref{thm:main_1}. Let $B \subset P$ be bounded
  and $\varepsilon > 0$. We can choose $N \in \NN$ such that
  \begin{equation*}
    d\big(\Phi(u),F_u^n(x_0)\big) \leq \frac{K C^n}{1-C} < \frac{\varepsilon}{2L}
  \end{equation*}
  for all $u \in B$ and all $n \geq N$, where $L > 0$ denotes a Lipschitz constant of $G$. Moreover, assumption (d)
  implies that all elements of the recursively defined sequence
  \begin{equation}
    \label{eq:recursion-2}
    W_0(x_0) := \{x_0\} \quad \text{and} \quad W_{n+1}(x_0) := F(W_n \times B)\,.
  \end{equation}
  become relatively compact when applying $G$. Therefore, the set $\{G\big(F^N_u(x_0)\big)  :  u\in B\} \subset G\big(W_N(x_0)\big)$
  is also relatively compact and thus admits a finite $\frac{\varepsilon}{2}$-net $z_1, \dots, z_M$. Hence, for all $u \in B$ we can find
  $z_i$ such that $d\big(G\big(F^N_u(x_0)\big),z_i\big) < \frac{\varepsilon}{2}$ and consequently one has
  \begin{equation*}
    \begin{split}
      d\big(G\big(\Phi(u)\big),z_i\big) & \leq d\big(G\big(\Phi(u)\big),G\big(F_u^N(x_0)\big)\big) + d\big(G\big(F_u^N(x_0)\big),z_i\big)\\
      & < L d\big(\Phi(u),F_u^N(x_0)\big) + \frac{\varepsilon}{2} < \varepsilon\,,
    \end{split}
  \end{equation*}
  i.e.~$G\big(\Phi(B)\big)$ is relatively compact as $z_1,\ldots,z_M$ constitutes a finite $\varepsilon$-net.
\end{proof}

\noindent
Obviously, one can combine Corollary \ref{cor:main_1} or \ref{cor:main_2} with Corollary \ref{cor:main_3a}; yet it seems to
be superfluous to list these straightforward modifications here.

\section{Applications and the Ball-Marsden-Slemrod Conjecture}
\label{sec:application}

Next, we want to apply the previous findings (i) to provide a simplified approach to the non-controllability result
by Ball, Marsden, and Slemrod for $p > 1$ and (ii) to derive a proof of the Ball-Marsden-Slemrod conjecture for $p=1$,
cf.~Subsection \ref{subsec:application_p>1} and Subsection \ref{subsec:application_p=1}, respectively. In doing
so, we generalize earlier results in \cite{BMS1982} and \cite{BCC2019b,BCC2019a,BCC2020} from bilinear systems
to semi-linear ones by allowing non-linear vector fields $f_i$ which are Lipschitz on bounded sets and linearly
bounded. A first result in the same direction can be found in \cite{ChTh2019}.

\medskip
\noindent

Before going into the details, a few clues are advisable to ``maneuver'' the reader through the following
  two subsections:
  While the approach of Subsection \ref{subsec:application_p>1} is completely based on Corollary \ref{cor:main_2}, the situation
  in Subsection \ref{subsec:application_p=1} is bit more subtle. For proving the existence of a unique fixed point map we
  will apply Corollary \ref{cor:main_1} because the $\omega$-norm trick of Section \ref{subsec:application_p>1} does not
  work for $p=1$. Therefore, we have to estimate the Lipschitz constants of higher powers of the integral operator
  \eqref{eq:integral-operator} (see Lemma \ref{lem5:app-contraction} below) in order to see that it is an eventual
  contraction\footnote{In
      principle, by Lemma \ref{lem:equivalent-metric} both approaches are equivalent. Yet, the equivalent norm constructed
      in Lemma \ref{lem:equivalent-metric} cannot be expressed explictly as an $\omega$-norm. In that sense the second approach
      is stronger than the first one (see also Remark \ref{rem:app-contraction}).}.
  Once we have solved this problem we are faced with the actual
  challenge: the collapse of Lemma \ref{lem1:app-compactness} which is revealed in the non-compactness of the fixed point
  map $\Phi: u \mapsto x(\cdot, \xi_0,u)$ for $p=1$, cf.~Appendix \ref{sec:counter-example}. Here, Corollary \ref{cor:main_3a}
  comes into play and allows to prove relative compactness of $\cR^1(\xi_0)$ directly, i.e.~without Lemma
  \ref{lem2:app-compactness}, see also Remark \ref{rem:p=1}.
 
\medskip
\noindent
Throughout Section \ref{sec:application}, we assume that the infinitesimal generator $A: D(A) \to X$ is of
class $(M,\mu)$ with $M \geq 1$, $\mu \geq 0$, i.e.~for all $t \geq 0$ one has the estimate
\begin{equation*}
  \|\re^{tA}\| \leq M\re^{\mu t}\,.
\end{equation*}
Note that every infinitesimal generator of a $C^0$-semigroup is of some class $(M',\mu')$ with $M' \geq 1$, $\mu' \geq 0$
\cite[Chap.~1, Thm.~2.2]{Pazy1983}. Hence requiring $A$ to be of class $(M,\mu)$ serves only to fix the constants $M \geq 1$
and $\mu \geq 0$ but does not pose an actual restriction.


\medskip
\noindent
Moreover, we say that a vector field $f:\RR_0^+ \times X \to X$ satisfies Assumptions (A1) and (A2), respectively,
if the following conditions are fulfilled.

\medskip
\noindent
\textbf{Assumption A1:}
$f:\RR_0^+ \times X \to X$ is Lipschitz in $\xi$ on bounded sets with $L_{\rm loc}^\infty$-Lipschitz rate,
i.e.~for every bounded set $C \subset X$ there exists  $L \in L_{\rm loc}^\infty\big(\RR_0^+,\RR_0^+\big)$ such that
\begin{equation*}
  \big\|f(t,\xi) - f(t,\eta)\big\| \leq L(t) \|\xi - \eta\|
\end{equation*}
for all $\xi,\eta \in C$ and all $t \in \RR_0^+$.

\medskip
\noindent
\textbf{Assumption A2:}
$f:\RR_0^+ \times X \to X$ is linearly bounded, i.e.~there exist $\alpha,\beta \in L_{\rm loc}^\infty\big(\RR_0^+,\RR_0^+\big)$
such that
\begin{equation}
\label{eq:linearly-bounded}
  \big\|f(t,\xi) \big\| \leq \alpha(t)\|\xi(t)\| + \beta(t)
\end{equation}
for all $\xi \in X$ and all $t\in\RR_0^+$.

\medskip
\noindent
Finally, for $T \geq 0$ and continuous $f_i: \RR_0^+ \times X \to X$, $i = 1, \dots, m$ we define the integral operator
\begin{equation*}
  F: C\big([0,T], X\big) \times L^p\big([0,T],\RR^m\big) \to C\big([0,T], X\big)
\end{equation*}
\begin{equation}
  \label{eq:integral-operator}\tag{IO}
  F(x,u)(t) := \re^{tA}\xi_0 + \sum_{i=1}^m \int_0^t \re^{(t-s)A} u_i(s) f_i\big(s,x(s)\big) \,\rd s\,.
\end{equation}


\subsection{The case $p > 1$}
\label{subsec:application_p>1}

\begin{theorem}
  \label{thm:app_p>1}
  Let $X$ be a Banach space, let $A$ be the infinitesimal generator of a $C^0$-semigroup $\big(\re^{tA}\big)_{t \geq 0}$
  of bounded operators on $X$, and let $p > 1$. Moreover, let $f_i: \RR_0^+ \times X \to X$, $i = 1, \dots, m$ be
  continuous vector fields which satisfy Assumptions (A1) and (A2). Then for all $T \geq 0$, all $\xi_0 \in X$, and
  all $u \in L^p\big([0,T], \RR^m\big)$ the equation
  \begin{equation}
    \label{eq:unique_solution}
    \dot{x}(t) = Ax(t) + \sum_{i=1}^m u_i(s) f_i\big(t,x(t)\big)\,, \quad x(0) = \xi_0
  \end{equation}
  has a unique mild solution and the solution operator
  \begin{equation*}
    \Phi:L^p\big([0,T], \RR^m\big) \to C\big([0,T],X\big)
  \end{equation*}
  which assigns to each control $u \in L^p\big([0,T], \RR^m\big)$ the unique mild solution of \eqref{eq:unique_solution}
  is compact and locally Lipschitz continuous.
\end{theorem}

\begin{remark}
  \begin{enumerate}
  \item[(a)]
    The assumption $L_i \in L_{\rm loc}^\infty\big(\RR_0^+,\RR_0^+\big)$ in Theorem \ref{thm:app_p>1}, cf.~Assumption (A1),
    can be readily improved by $L_i \in L_{\rm loc}^{q'}\big(\RR_0^+,\RR\big)$ with $q' > q$, where $q$ denotes the conjugate
    exponent to $p$, i.e.~$\frac{1}{p} +\frac{1}{q} =1$. The case $q' = q$ is more delicate, because here the proof of
    Lemma \ref{lem3:app-contraction} below will break down. Actually, the case $q' = q$ resembles case $p=1$ of Subsection
    \ref{subsec:application_p=1} and therefore the technique used in Lemma \ref{lem5:app-contraction} can be successfully
    employed to prove at least existence and uniqueness of solutions. A similar comment applies to
    $\alpha_i,\beta_i \in L_{\rm loc}^\infty\big(\RR_0^+,\RR_0^+\big)$, cf.~Assumption (A2). \smallskip
  \item[(b)]
    Moreover, we expect that the continuity assumption on $f_i$ can be weakened (with respect to $t$) in the sense
    of Carath\'{e}odory's theorem, cf.~e.g.~\cite{CodLev87,Sontag98}.\smallskip
  \item[(c)]
    Conditions which guarantee that mild solutions are actually classical or Cara\-th\'{e}odory solutions can be found in
    \cite[Sec.~6.1]{Pazy1983}.\smallskip
  \item[(d)]
    The existence and uniqueness result of Ball, Marsden, and Slemrod given in \cite[Thm.~2.5]{BMS1982} is stronger than
    that in Theorem \ref{thm:app_p>1} because there Lipschitz continuity of the non-linear part is only required locally
    and not on bounded sets. However, their non-controllability result \cite[Thm.~3.6]{BMS1982} is considerably
      weaker than Corollary \ref{cor:app_p>1} below as it is restricted to bilinear systems, i.e.~to bounded linear
    vector fields $f_i$. Related results also dealing with non-linear $f_i$ can be found in \cite{ChTh2019,ChTh2020}.
  \end{enumerate}
\end{remark}

\noindent
As an immediate consequence of Theorem \ref{thm:app_p>1} and Lemma \ref{lem2:app-compactness} below we obtain
the following generalization of Ball-Marsden-Slemrod's non-controllability statement.

\begin{corollary}
  \label{cor:app_p>1}
  Let the notation and assumptions be as in Theorem \ref{thm:app_p>1}. Moreover, assume that all $f_i$, $i = 1, \dots, m $
  are autonomous. Then for all $\xi_0 \in X$ and all $p > 1$ the reachable set $\cR^p(\xi_0)$ of
  \begin{equation}
  \label{eq:bilin_control_p>1}
  \dot{x}(t) = Ax(t) + \sum_{i=1}^m u_i(t) f_i\big(x(t)\big)\,,
  \quad x(0) = \xi_0\,, \quad u_i \in L^p_{\rm loc}\big([0,\infty), \RR\big)
  \end{equation}
  can be written as a countable union of relatively compact sets and has therefore no interior points if $X$ is
  infinite-dimensional.
\end{corollary}

\noindent
Here we assumed that the vector fields $f_i$ are autonomous since this is the standard setting in control theory. Of course,
the result holds for time-dependent vector fields as well---but then $\cR^p(\xi_0)$ depends in general additionally on the initial
time of \eqref{eq:bilin_control_p>1} and would not be called reachable set of $\xi_0$.

\begin{proof}[Proof of Corollary \ref{cor:app_p>1}]
  Theorem \ref{thm:app_p>1} and Lemma \ref{lem2:app-compactness} below immediately imply that the reachable set
  up to time $T$ (under bounded controls $\|u\|_p \leq r$), i.e.
  \begin{equation*}
    \cR^{p,r}_{\leq T}(\xi_0) := \{x(t,\xi_0,u)  :  t \in [0,T]\,, \|u\|_p \leq r\}\,,
  \end{equation*}
  is relatively compact. Moreover, the identity
  \begin{equation*}
    \cR^p(\xi_0) = \bigcup_{T \in \NN} \bigcup_{n \in \NN} \cR^{p,n}_{\leq T}(\xi_0)
  \end{equation*}
  shows that the reachable set $\cR^p(\xi_0)$ is a countable union of relatively compact sets. Finally, a straightforward application
  of Baire's category theorem combined with fact that relatively compact sets in an infinite-dimensional Banach space are nowhere
  dense shows that $\cR^p(\xi_0)$ has no interior points.
\end{proof}

\begin{lemma}
  \label{lem2:app-compactness}
  Let $S \subset C\big([0,T], X\big)$ be relatively compact. Then the evaluation set
  ${\rm Ev}(S) := \{x(t)  :  x \in S\,, t \in [0,T]\} \subset X$ is relatively compact, as well.
\end{lemma}

\begin{proof}
  Let $\varepsilon > 0$. Since $S \subset C\big([0,T], X\big)$ is relatively compact by assumption, we can choose
  a finite $\frac{\varepsilon}{2}$-net $x_1, \dots x_N $ of $S$. 
  Moreover, compactness of $[0,T]$ implies that the images of $x_i$, $i=1,\dots,N$ are also
  compact. Hence we can cover all $x_i\big([0,T]\big)$ and their union by a finite $\frac{\varepsilon}{2}$-net
  $\xi_1, \dots \xi_M \in X$.
  Now, for $x \in S$ and $t \in [0,T]$ we can choose $x_i$ and $\xi_j$ such that
  \begin{equation*}
    \|x(t) - \xi_j\| \leq \|x(t) - x_i(t)\| + \|x_i(t) - \xi_j\|
    < \|x - x_i\|_\infty + \frac{\varepsilon}{2} < \varepsilon\,.
  \end{equation*}
  This shows that $\xi_j$, $j = 1, \dots, M$ yields a finite $\varepsilon$-net of ${\rm Ev}(S)$ and thus ${\rm Ev}(S)$
  is relatively compact.
\end{proof}

\begin{remark}
  A minor modification of Lemma 3.7 in \cite{BMS1982} to non-linear maps satisfying Assumption (A1) allows to show
  that the integral operator $F$ given by \eqref{eq:integral-operator} is weakly continuous in $u$ and thus the solution
  operator $u \to \Phi(u)$ will be weakly continuous, too. Hence for $p >1$ the sets $\cR^{p,r}_{\leq T}(\xi_0)$
  are even compact and not only relatively compact.
\end{remark}

In order to prove Theorem \ref{thm:app_p>1}, we will show that the integral operator $F$ given by \eqref{eq:integral-operator}
satisfies the conditions of Theorem \ref{thm:main_1} or, more precisely, of Corollary \ref{cor:main_2}. To this end
we define on $C\big([0,T], X\big)$ the weighted $\omega$-norms
\begin{equation*}
  \|x\|_{\omega} := \max_{t \in [0,T]}\re^{-\omega t}\|x(t)\|\,.
\end{equation*}
Obviously, $\|\cdot\|_{0}$ coincides with the standard maximum norm on $C\big([0,T], X\big)$ and all $\omega$-norms
are equivalent on $C\big([0,T], X\big)$. To avoid confusion, whenever necessary, we will write $C\big([0,T], X\big)_\omega$
to indicate that a particular statement about $C\big([0,T], X\big)$ holds (only) with respect to a particular
$\omega$-norm.

\medskip
\noindent
\textbf{Note:} For simplicity, we will prove the following auxiliary results for the case $m=1$. This is justified because
the case $m > 1$ can be handled completely analogously.

\begin{lemma}
  \label{lem1:app-continuity}
  Let $f_i$, $i = 1, \dots, m$ be continuous and $p \geq 1$. Then the operator
  $F: C\big([0,T], X\big) \times L^p\big([0,T],\RR^m\big) \to C\big([0,T], X\big)$ as given in \eqref{eq:integral-operator}
  is globally Lipschitz in $u$ locally uniformly in $x$.
\end{lemma}

\begin{proof}
  Global Lipschitz continuity of $F_x$ readily follows from the estimate
  \begin{equation}
    \label{eq:Lipschitz_in_u}
    \begin{split}
      \big\|  F_x(v) - F_x(u)\big\|_\infty
      &= \max_{t \in [0,T]} \bigg\| \int_0^t \re^{(t-s)A} \big(v(s)-u(s)\big) f\big(s,x(s)\big) \,\rd s\bigg\|\\
      & \leq \max_{t \in [0,T]} \int_0^t M \re^{\mu (t-s)} \big|v(s)-u(s)\big| \Big\|f\big(s,x(s)\big)\Big\| \,\rd s\\
      & \leq K M \re^{\mu T} \int_0^T \big|v(s)-u(s)\big| \,\rd s\\
      & \leq
      \begin{cases}
        K M \re^{\mu T} \big\|v - u\big\|_1  & \text{for $p = 1$,}\\[2mm]
        K M \re^{\mu T} T^{1/q} \big\|v - u\big\|_p  & \text{for $p > 1$,}
      \end{cases}
    \end{split}
  \end{equation}
  where $K \geq 0$ is defined by $K := \max_{t \in [0,T]} \big\|f(t,x(t))\big\|$
  and $T^{1/q}$ results from H\"{o}lder's inequality (with $\frac{1}{p} + \frac{1}{q} = 1$).

  \medskip
  \noindent
  To see that the above estimate can be made locally uniformly in $x$ one can proceed as follows: Due to continuity
  of $f$ there exist $\delta_t>0$ and $\varepsilon_t>0$ such that
  \begin{equation*}
    \big\|f(s,\eta) - f\big(t,x(t)\big)\big\| \leq 1
    \quad \text{for all $s \in (t-\delta_t,t+\delta_t)$ and all $\eta \in B_{\varepsilon_t}\big(x(t)\big)$.}
  \end{equation*}
  Now the collection of all open balls $B_{\varepsilon_t/2}\big(x(t)\big)$, $t \in [0,T]$ yields an open cover of the
  compact set $x\big([0,T]\big)$ and consequently there exists a finite subcover
  \begin{equation*}
    B_{\varepsilon_1/2}\big(x(t_1)\big), \dots, B_{\varepsilon_N/2}\big(x(t_N)\big) \quad \text{with $\varepsilon_i :=\varepsilon_{t_i} $.}
  \end{equation*}
  Set $\varepsilon^* := \min\{\varepsilon_1, \dots, \varepsilon_N\}$. Then for all  $t \in [0,T]$ one can choose $t_i\in [0,T]$
  such that $\|x(t) - x(t_i)\| < \varepsilon_i/2$ and hence for all $y \in B_{\varepsilon^*/2}(x) \subset C\big([0,T], X\big)$ one has
  \begin{equation*}
    \|y(t) - x(t_i)\| \leq \|y(t) - x(t)\| + \|x(t) - x(t_i)\| < \varepsilon_i\,.
  \end{equation*}
  This implies
  \begin{equation*}
    \big\|f(t,y(t))\big\| \leq \big\|f(t,y(t)) - f(t,x(t_i))\big\| + \big\|f(t,x(t_i))\big\| \leq 1 + K
  \end{equation*}
  and thus
  \begin{equation*}
    \sup_{y \in B_{\varepsilon^*/2}(x)}\max_{t \in [0,T]} \big\|f(t,y(t))\big\| \leq 1 + K\,.
  \end{equation*}
  It follows that estimate \eqref{eq:Lipschitz_in_u} is even locally uniformly in $x$ once $K$ is replaced by $K+1$.
\end{proof}

\begin{remark}
    One can significantly simplify the second part of the above proof by assuming that all $f_i$, $i = 1, \dots, m$
    satisfy an additional Lipschitz condition in $\xi$ as in Theorem \ref{thm:app_p>1}.
\end{remark}

\begin{lemma}
  \label{lem1:app-compactness}
  Let $f_i$, $i = 1, \dots, m$ be continuous and $p > 1$. Then for all $x \in C\big([0,T], X\big)$ the operator
  $F_x := F(x,\,\cdot\,): L^p\big([0,T],\RR^m\big) \to C\big([0,T], X\big)$ is compact.
\end{lemma}

\begin{proof}
  Without loss of generality let $B := K_1(0)$ be the closed unit ball of $L^p\big([0,T],\RR\big)$ and let
  $\varepsilon > 0$. We have to find a finite $\varepsilon$-net for the image $F_x\big(B\big) \subset C\big([0,T], X\big)$.
  To this end, we consider the set $K := \big\{f\big(t,x(t)\big) \;:\; t \in [0,T]\big\}$ which is obviously
    compact as $x(\cdot)$ and $f$ are continuous. By Lemma \ref{lem3:compactness_p=1} one can choose disjoint
    Borel sets $\Delta_i \subset [0,T]$, $i = 1, \dots, N$ and $S_j \subset K$, $j = 1, \dots, M$ as well as $\xi_{i,j} \in X$
    such that the  uniform approximation
    \begin{equation*}
      \Big\|\re^{tA}\xi - \sum_{i=1}^N\sum_{j=1}^M\chi_{\Delta_i}(t) \chi_{S_j}(\xi)\xi_{ij}\Big\| \leq \frac{\varepsilon}{2 T^{1/q}}
    \end{equation*}
    holds for all $t \in [0,T]$ and all $\xi \in K$. It follows
    \begin{equation*}
      \begin{split}
        \max_{t \in [0,T]}&\Big\|F_x(u)(t) -\Big( \re^{tA}\xi_0 + \int_0^tu(s)\sum_{i=1}^N\sum_{j=1}^M\chi_{\Delta_i}(t-s) \chi_{S_j}\big(f(s,x(s))\big)\xi_{ij} \,\rd s\Big)\Big\| \\
      &  \leq \max_{t \in [0,T]}\int_0^t \big|u(s)\big| \Big\| \re^{(t-s)A}f(s,x(s)) - \sum_{i=1}^N\sum_{j=1}^M\chi_{\Delta_i}(t-s) \chi_{S_j}\big(f(s,x(s))\big)\xi_{ij}\Big\| \,\rd s\\
      &\leq \frac{\varepsilon}{2T^{1/q}} \int_0^T \big|u(s)\big| \,\rd s \leq \frac{\varepsilon\|u\|_p}{2} \leq \frac{\varepsilon}{2}\,,
    \end{split}
  \end{equation*}
  for all $u \in K_1(0)$, where we used H\"older's inequality in the second-to-last step. Thus it suffices to show that the set
  \begin{equation*}
    \cS := \Bigg\{t \mapsto \int_0^t u(s)\sum_{i=1}^N\sum_{j=1}^M\chi_{\Delta_i}(t-s) \chi_{S_j}\big(f(s,x(s))\big)\xi_{ij}\,\rd s  : u \in B \Bigg\} \subset C([0,T],X)
  \end{equation*}
  has a finite $\frac{\varepsilon}{2}$-net $y_{1}, \dots, y_{N}$ because then the previous estimates guarantee that
  $y_{1}, \dots, y_{N}$  translated by $y_0: t \mapsto \re^{tA}\xi_0$ yields an $\varepsilon$-net of $F_x(B)$. Note that
  integrability of $s \mapsto \sum_{i=1}^N\sum_{j=1}^M\chi_{\Delta_i}(t-s) \chi_{S_j}\big(f(s,x(s))\big)\xi_{ij}$ follows readily
  from its measurability and essential boundedness.
  Since $\cS$ is obviously bounded and all its functions take their values in the finite-dimensional subspace
    $X_\varepsilon := \linspan\{\xi_{ij} \;:\; i = 1, \dots, N\,, j = 1, \dots, M\}$ we can invoke Arzel\`{a}--Ascoli's theorem
    \cite[Thm.~7.25]{Rudin87} to prove relative compactness of $\cS$, that is, we have to show that $\cS$ is equicontinuous.
    To this end we consider 
  \begin{equation*}
    \begin{split}
      & \Big\|\int_0^{t+\delta} u(s) \sum_{i=1}^N\sum_{j=1}^M\chi_{\Delta_i}(t-s) \chi_{S_j}\big(f(s,x(s))\big)\xi_{ij}\,\rd s\\
      & \quad - \int_0^{t}u(s)  \sum_{i=1}^N\sum_{j=1}^M\chi_{\Delta_i}(t-s) \chi_{S_j}\big(f(s,x(s))\big)\xi_{ij} \,\rd s\Big\| \\
      & = \Big\|\int_t^{t+\delta}u(s) \sum_{i=1}^N\sum_{j=1}^M\chi_{\Delta_i}(t-s) \chi_{S_j}\big(f(s,x(s))\big)\xi_{ij} \,\rd s \Big\|
      \leq \max_{i = 1, \dots, N \atop j = 1, \dots, M}\|\xi_{ij}\| \int_t^{t+\delta}\big|u(s)\big|\,\rd s\\
      & \leq \max_{i = 1, \dots, N \atop j = 1, \dots, M}\|\xi_{ij}\| \delta^{1/q} \|u\|_p
      \leq \max_{i = 1, \dots, N \atop j = 1, \dots, M}\|\xi_{ij}\| \delta^{1/q}\,,
    \end{split}
  \end{equation*}
  where the second-to-last estimate follows again from H\"older's inequality. This clearly shows equicontinuity of $\cS$ and
  concludes the proof.
\end{proof}

\begin{lemma}
  \label{lem3:app-contraction}
  Let $f_i$, $i = 1, \dots, m$ be globally Lipschitz in $\xi$ with $L^\infty_{\rm loc}$-Lipschitz rate and let $p > 1$.
  Then there exists $\omega \geq 0$ such that
    \begin{equation*}
      F: C\big([0,T],X\big)_\omega \times L^p\big([0,T],\RR\big) \to C\big([0,T],X\big)_\omega
    \end{equation*}
    is a contraction in $x$ uniformly on bounded sets of $L^p([0,T],\RR)$.
\end{lemma}

\begin{proof}
  For simplicity, let $B$ be the unit ball of $L^p([0,T],\RR)$. Then for $x,y \in C\big([0,T],X\big)$, $u \in B$, and
  $\omega \geq 0$ one has the estimate
  \begin{equation*}
    \begin{split}
      \big\|F(x,u) & - F(y,u)\big\|_{\omega}
      = \max_{t \in [0,T]} \re^{-\omega t} \bigg\| \int_0^t \re^{(t-s)A}u(s)\big(f(s,x(s))-f(s,y(s))\big) \,\rd s\bigg\|\\
      & \leq \max_{t \in [0,T]} \re^{-\omega t} \int_0^t |u(s)| \, \big\|\re^{(t-s)A}\big\| \, \big\|f(s,x(s))-f(s,y(s))\big\| \,\rd s\\
      & \leq \max_{t \in [0,T]} \re^{-\omega t}
      \int_0^t M \re^{\mu(t-s)} \re^{\omega s} \, |u(s)|  \, L(s) \, \re^{- \omega s} \, \big\|x(s) - y(s)\big\| \,\rd s\\
      & \leq \max_{t \in [0,T]}  M \|L\|_\infty \re^{(\mu -\omega) t} \int_0^t \re^{(\omega-\mu)s} |u(s)| \,\rd s \, \big\|x - y \big\|_\omega\\
      & \leq \max_{t \in [0,T]}  M \|L\|_\infty \re^{(\mu -\omega) t} \big\|u\big\|_p
      \Big(\int_0^t\re^{q(\omega-\mu)s} \,\rd s \Big)^{1/q}\, \big\|x - y \big\|_\omega\\
      & \leq \max_{t \in [0,T]}  M \|L\|_\infty \re^{(\mu -\omega) t} 
      \bigg(\Big[\frac{\re^{q(\omega-\mu)s}}{q(\omega-\mu)}\Big]_0^t \bigg)^{1/q}\, \big\|x - y \big\|_\omega\,,
    \end{split}
  \end{equation*}
  where $\|L\|_\infty$ denotes the $L^\infty$-norm of $L$ on $[0,T]$ and the second-to-last estimate
  exploits again H\"older's inequality. Finally, for $\omega > \mu$ we obtain
  \begin{equation*}
    \begin{split}
      \big\|F(x,u) - F(y,u)\big\|_{\omega}& \leq \max_{t \in [0,T]}  M \|L\|_\infty \re^{(\mu -\omega) t}
      \bigg( \frac{\re^{q(\omega-\mu)t}-1}{q(\omega-\mu)} \bigg)^{1/q}\, \big\|x - y \big\|_\omega\\
      & \leq \frac{M \|L\|_\infty }{\big(q(\omega-\mu)\big)^{1/q}}\, \big\|x - y \big\|_\omega\,.
    \end{split}
  \end{equation*}
  Thus, for $\omega > 0$ sufficiently large, we can achieve $\frac{M \|L\|_\infty }{(q(\omega-\mu))^{1/q}} < 1$.
  This completes the proof.
\end{proof}

\noindent
The following technical result provides the justification that, under Assumptions (A1) and (A2), one can
actually assume that all $f_i$ are globally Lipschitz in $\xi$ which in turn allows to apply Lemma \ref{lem3:app-contraction}. 

\begin{lemma}
  \label{lem:cutoff}
  Let $f:\RR_0^+ \times X \to X$ satisfy Assumptions (A1) and (A2) and let $p \geq 1$.
  Then for every bounded set $B \subset L^p\big([0,T],\RR\big)$ and every $\xi_0 \in X$ there exists
  $\hat{f}:\RR_0^+ \times X \to X$ which is globally Lipschitz in $\xi$ with $ L_{\rm loc}^\infty$-Lipschitz rate
  such that every mild solution of
  \begin{equation}
    \label{eq:f_without_cutoff}
    \dot{x}(t) = Ax(t) + u(s) f\big(t,x(t)\big)\,, \quad x(0) = \xi_0\,, \quad u \in B
  \end{equation}
  on $[0,T]$ is a mild solution of 
  \begin{equation}
    \label{eq:f_with_cutoff}
    \dot{x}(t) = Ax(t) + u(s) \hat{f}\big(t,x(t)\big)\,, \quad x(0) = \xi_0\,, \quad u \in B
  \end{equation}
  on $[0,T]$ and vice versa.
\end{lemma}

\begin{proof}
  For $n \in \NN$ let $\rho_n: X \to [0,1]$ be a globally Lipschitz cut-off function which is equal to $1$ for
  $\|\eta\| \leq n$ and zero for $\|\eta\| \geq n+1$, for instance
  \begin{equation*}
    \rho_n(\eta) :=
    \begin{cases}
      1 & \text{for $\|\eta\| \leq n$,}\\
      0 & \text{for $\|\eta\| \geq n+1$,}\\
      n+1 - \|x\| & \text{for $n \leq \|\eta\| \leq n+1$.}
    \end{cases}
  \end{equation*}
  The idea is, as $t \geq 0$ is restricted to the compact interval $[0,T]$, to provide a priori estimates (via
  Gronwall's lemma) for the solutions of \eqref{eq:f_without_cutoff} and \eqref{eq:f_with_cutoff} such that
  $\hat{f}$ can be defined as $\hat{f}(t,\eta) := \rho_n(\eta-\xi_0) f(t,\eta)$ for some appropriate $n \in \NN$.
  In other words we want to choose $n \in \NN$ such that the difference between $f$ and $\hat f$ occurs outside
  of an appropriate ball $B_n(\xi_0)$ which contains all mild solutions of the initial value
  problems \eqref{eq:f_without_cutoff} and \eqref{eq:f_with_cutoff}.

  \medskip
  \noindent
  To this end we assume that $x:[0,T] \to X$ is a mild solution of \eqref{eq:f_without_cutoff} and define
  $\varphi(t) := \re^{-\mu t}\|x(t) - \re^{tA}\xi_0\|$. Then one has
  \begin{equation*}
    \begin{split}
      \varphi(t) & = \re^{-\mu t}\bigg\|\int_0^t \re^{(t-s)A}u(s) f\big(s,x(s)\big) \,\rd s\bigg\|
      \leq \int_0^t M \re^{-\mu s}|u(s)| \Big(\alpha(s) \|x(s)\|+ \beta(s)\Big)  \,\rd s \\
      & \leq M \int_0^t |u(s)| \Big(\alpha(s) \re^{- \mu s}\big\|x(s) - \re^{sA}\xi_0 + \re^{sA}\xi_0\big\|+ \re^{-\mu s}\beta(s)\Big)  \,\rd s \\
      & \leq M \int_0^t |u(s)| \Big(\alpha(s) \varphi(s) + \alpha(s) \re^{-\mu s}\big\|\re^{sA}\xi_0\big\|+ \re^{-\mu s}\beta(s)\Big)  \,\rd s \\
      & \leq M \int_0^t \alpha(s) |u(s)| \varphi(s) \,\rd s
      + M \int_0^t |u(s)| \Big(\alpha(s)  \re^{-\mu s}\|\re^{sA}\xi_0\| + \re^{-\mu s}\beta(s)\Big) \,\rd s \\
      & \leq M \|\alpha\|_\infty \int_0^t |u(s)| \varphi(s) \,\rd s + M C_p \|u\|_p \big(M\|\alpha\|_\infty  \|\xi_0\| + \|\beta\|_\infty \big)\,,
    \end{split}   
  \end{equation*}
where the constant $C_p := T^{(p-1)/p}$ results from H\"older's inequality
.
  A straightforward application of Gronwall's lemma \cite[App.~C]{Sontag98} yields
  \begin{equation*}
    \begin{split}
      \varphi(t) \leq M C_p \|u\|_p \big(M \|\alpha\|_\infty  \|\xi_0\| + \|b\|_\infty \big)\re^{M \|\alpha\|_\infty \int_0^t |u(s)| \,\rd s}
    \end{split}   
  \end{equation*}
  and thus
  \begin{equation*}
    \begin{split}
      \big\|x(t) & - \xi_0\big\| \leq \re^{\mu t} \varphi(t) + \|\re^{tA}\xi_0 - \xi_0\|\\
      & \leq  M\re^{\mu t} C_p \|u\|_p \big(M \|\alpha\|_\infty \|\xi_0\| + \|\beta\|_\infty \big) \re^{M \|\alpha\|_\infty \int_0^t |u(s)| \,\rd s}
      + (M \re^{\mu t}+1)\|\xi_0\|\\
      & \leq M\re^{\mu T} \Big( C_p K \big(M \|\alpha\|_\infty  \|\xi_0\| + \|\beta\|_\infty \big)\re^{M \|\alpha\|_\infty C_p K} + 2\|\xi_0\|\Big)
      =: R\,,
    \end{split}   
  \end{equation*}
  where the last estimate exploits the assumption that $B \subset L^p\big([0,T],\RR\big)$ is bounded, i.e.~$\|u\|_p \leq K$
  for some constant $K \geq 0$. Next let us choose $N \in \NN$ with $N > R$ and define $\hat{f}(t,\eta) := \rho_N(\eta-\xi_0) f(t,\eta)$.
  Then, by the above estimate, $x(t)$ is also a mild solution of \eqref{eq:f_with_cutoff}. Conversely, since $\hat{f}$
  also satisfies \eqref{eq:linearly-bounded} we obtain the same a priori estimate for any solution $\hat{x}(t)$ of
  \eqref{eq:f_with_cutoff}, and thus $\hat{x}(t)$ is also a solution of \eqref{eq:f_without_cutoff}.

  \medskip
  \noindent
  Finally, global Lipschitz continuity of $\hat{f}$ results from the computation
  \begin{equation*}
    \begin{split}
       \big\|\hat{f}&(t,\eta) - \hat{f}(t,\zeta)\big\| \\
      &\leq
      \big|\rho_N(\eta-\xi_0)\big|\,\big\|f(t,\eta) - f(t,\zeta)\big\| + \big|\rho_N(\eta-\xi_0) - \rho_N(\zeta-\xi_0)\big|\,\big\|f(t,\zeta)\big\|\\[2mm]
      & \leq \big|\rho_N(\eta-\xi_0)\big|\,L(t)\|\eta - \zeta\| + \big|\rho_N(\eta-\xi_0) - \rho_N(\zeta-\xi_0)\big| \big(\alpha(t) \|\zeta\|+\beta(t)\big)\\[2mm]
& \leq \big|\rho_N(\eta-\xi_0)\big|\,L(t)\|\eta - \zeta\| + \big(\alpha(t) \|\zeta\|+\beta(t)\big)\,\|\eta - \zeta\|\\[2mm]
      & \leq
      \begin{cases}
        0 &\|\eta-\xi_0\| ,\|\zeta-\xi_0\| \geq N+1\,,\\[2mm]
        \big(L(t)+ \alpha(t) \big(N+1+\|\xi_0\|\big) + \beta(t) \big)\|\eta - \zeta\| &\|\eta-\xi_0\| ,\|\zeta-\xi_0\| < N+1\,,\\[2mm]
        \big(\alpha(t) \big(N+1+\|\xi_0\|\big) + \beta(t)\big)\|\eta - \zeta\| & \|\zeta-\xi_0\| < N+1 \leq \|\eta-\xi_0\|\,, 
      \end{cases}
    \end{split}   
  \end{equation*}
  where we used the fact that $\rho_N$ has Lipschitz rate $L=1$. Interchanging the role of $\eta$ and $\zeta$
  yields the corresponding estimate for $\|\eta-\xi_0\| < N+1 \leq \|\zeta-\xi_0\|$.
\end{proof}

\begin{proof}[Proof of Theorem \ref{thm:app_p>1}]
  Let $T \geq 0$ and $p > 1$ be fixed in the following. Obviously, if $x:[0,T] \to X$ is a mild solution of
  \eqref{eq:unique_solution} to the control $u \in L^p\big([0,T],\RR^m\big)$
  it is a fixed point of the integral operator \eqref{eq:integral-operator}
  \begin{equation*}
    F: C\big([0,T], X\big) \times L^p\big([0,T],\RR^m\big) \to C\big([0,T], X\big)
  \end{equation*}
  given by
  \begin{equation*}
    F(x,u)(t) = \re^{tA}\xi_0 + \sum_{i=1}^m\int_0^t \re^{(t-s)A}u_i(s)f_i\big(s,x(s)\big)\,\rd s\,.
  \end{equation*}
  Therefore, our aim is to apply Corollary \ref{cor:main_2} to $F$. To this end, we first restrict
  $F$ to $C\big([0,T], X\big) \times K_r(0)$, where $K_r(0)$ denotes an arbitrary closed
  ball in $L^p\big([0,T],\RR^m\big)$. 
  Consequently, due to Lemma \ref{lem:cutoff} we can further assume that all $f_i:\RR_0^+ \times X \to X$
  are globally Lipschitz in $\xi$ with $L^\infty_{\rm loc}$-Lipschitz rate. 

  \medskip
  \noindent
  Hence condition (a) of Corollary \ref{cor:main_2} follows from Lemma \ref{lem3:app-contraction}; conditions (b)
  and (d) are implied by Lemma \ref{lem1:app-continuity} and condition (c) by Lemma \ref{lem1:app-compactness}.
  Thus the fixed point map $\Phi:u \mapsto x(\cdot,\xi,u)$ is compact and locally Lipschitz continuous on $K_r(0)$. 
  Since $r > 0$ can be chosen arbitrarily the desired result follows.
\end{proof}



\subsection{The case $p = 1$}
\label{subsec:application_p=1}

\begin{theorem}
  \label{thm:app_p=1}
  Let $X$ be a Banach space and let $A$ be the infinitesimal generator of a $C^0$-semigroup $\big(\re^{tA}\big)_{t \geq 0}$
  of bounded operators on $X$. Moreover, let $f_i: X \to X$, $i = 1, \dots, m$ be continuous vector fields which satisfy
  Assumptions (A1) and (A2).
  Then for all $\xi_0 \in X$ the reachable set $\cR^1(\xi_0)$ of
  \begin{equation*}
    \dot{x}(t) = Ax(t) + \sum_{i=1}^m u_i(t) f_i(t,x(t))\,, \quad x(0) = \xi_0\,,
    \quad u_i \in L^1_{\rm loc}\big([0,\infty), \RR\big)
  \end{equation*}
  can be written as a countable union of relatively compact sets and therefore has no interior points if $X$
  is infinite-dimensional.
\end{theorem}

\begin{remark}
  \label{rem:p=1}
  Note that Lemma \ref{lem1:app-compactness} and Lemma \ref{lem3:app-contraction}---as used in the proof of Theorem
  \ref{thm:app_p>1}---explicitly require $p > 1$. Thus we have to fill these gaps in order to tackle the case $p=1$:
  Lemma \ref{lem3:app-contraction} can be readily replaced by Lemma \ref{lem5:app-contraction} below. However, Lemma
  \ref{lem1:app-compactness} is more delicate---actually, it cannot be repaired because the integral operator
  \eqref{eq:integral-operator} is in general not compact for $p=1$. A counter-example is provided in Appendix
  \ref{sec:counter-example}. However, compactness of the fixed point map $\Phi:u \mapsto x(\cdot,\xi_0,u)$ is
  not necessary for relative compactness of the reachable set $\cR^{1,r}_{\leq T}(\xi_0)$. Therefore, the goal is to
  prove relative compactness of $R^{1,r}_{\leq T}(\xi_0)$ directly via Corollary \ref{cor:main_3a}. Our strategy is
  closely related to the approach presented in \cite{BCC2019a}.
\end{remark}

\begin{lemma}
  \label{lem5:app-contraction}
  Let $f_i: \RR_0^+ \times X \to X$, $i = 1,\dots, m$ be globally Lipschitz in $\xi$ with
  $L^\infty_{\rm loc}$-Lipschitz rate. The iterated integral operator $F_u^n: C\big([0,T],X\big) \to C\big([0,T],X\big)$
  given by \eqref{eq:integral-operator} satisfies the estimate
  \begin{equation}
    \label{eq:app-contraction_p=1}
    \|F^n_u(x)(t) - F^n_u(y)(t)\|
    \leq \frac{M^n \re^{\mu n T} \|L\|^n_\infty \Big(\int_0^{t} v(s) \,\rd s\Big)^n \big\|x -y \big\|_\infty}{n!}\,.
  \end{equation}
  with $\|L\|_\infty := \max_{i = 1, \dots, m} \|L_i\|_\infty$ and $v(s) := \sum_{i=1}^m|u_i(s)|$. Moreover, 
  \begin{equation*}
    F: C\big([0,T],X\big) \times L^1\big([0,T],\RR^m\big) \to C\big([0,T],X\big)
  \end{equation*}
  is an eventual contraction in $x$ uniformly on bounded sets of $L^1\big([0,T],\RR^m\big)$.
\end{lemma}

\begin{proof}
  Let $x,y \in C\big([0,T],X\big)$, $u \in L^1\big([0,T],\RR\big)$ and suppose $m = 1$. Then we obtain
  \begin{equation*}
    \begin{split}
      \big\|F_u(x)(t) - F_u(y)(t)\big\| &=  \bigg\| \int_0^t \re^{(t-s)A}u(s)\big(f(s,x(s)) - f(s,y(s))\big) \,\rd s\bigg\|\\
      &  \leq \int_0^t M\re^{\mu(t-s)}|u(s)| L(s) \big\|x(s) - y(s)\big\| \,\rd s\\
      &  \leq M \re^{\mu T} \|L\|_\infty \int_0^t |u(s)| \,\rd s \, \big\|x - y\big\|_\infty 
    \end{split}
  \end{equation*}
  and thus \eqref{eq:app-contraction_p=1} holds for $n=1$. Next, assume that \eqref{eq:app-contraction_p=1} is satisfied
  for some $n \in \NN$. We conclude
  \begin{equation*}
    \begin{split} 
       \big\|F^{n+1}_u(x&)(t) - F^{n+1}_u(y)(t)\big\| = \big\|F_u\big(F^n_u(x)\big)(t) - F_u\big(F^n_u(y)\big)(t)\big\| \\
      & \quad = \bigg\| \int_0^t \re^{(t-s)A}u(s)\Big(f\big(s,F^n_u(x)(s)\big) - f\big(s,F^n_u(y)(s)\big)\Big) \,\rd s\bigg\|\\
      & \quad \leq \int_0^t M \re^{\mu (t-s)} |u(s)| \, L(s) \big\|F^n(x,u)(s) - F^n(y,u)(s)\big\| \,\rd s\\
      & \quad \leq \frac{M^{n+1} \re^{\mu (n+1) T} \|L\|^{n+1}_\infty \|x-y\|_\infty}{n!}
      \int_0^t |u(s)| \, \Big(\int_0^{s}|u(r)| \rd r\Big)^n \,\rd s\\
      & \quad = \frac{M^{n+1} \re^{\mu (n+1) T} \|L\|^{n+1}_\infty \|x-y\|_\infty}{(n+1)!} \Big(\int_0^{t}|u(s)| \,\rd s\Big)^{n+1}\,, 
    \end{split}
  \end{equation*}
  where the last step follows via integration-by-parts. Next, let $m \geq 1$ and choose $\|u\|_1 := \sum_{i=1}^m \|u_i\|_1$
  as norm\footnote{Obviously, this norm on $L^1\big([0,T],\RR^m\big)$ is induced by the $1$-norm on $\RR^n$. Choosing
    another norm on $\RR^n$ would simply yield an equivalent norm on $L^1\big([0,T],\RR^m\big)$.}
  on $L^1\big([0,T],\RR^m\big)$.
  It follows
  \begin{equation*}
    \begin{split}
      \|F_u(x)(t) - F_u(y)(t) \big\|
      \leq M \re^{\mu T} \|L\|_\infty  \int_0^t \sum_{i=1}^m |u_i(s)| \,\rd s \big\|x - y\big\|_\infty
    \end{split}
  \end{equation*}
  with $\|L\|_\infty := \max_{i = 1, \dots, m} \|L_i\|_\infty$. Setting $v(s) := \sum_{i=1}^m |u_i(s)|$, we conclude by
  induction
  \begin{equation*}
    \begin{split} 
      \|F^{n}_u(x)(t) - F^{n}_u(y)(t) \big\|
      \leq \frac{M^{n} \re^{\mu n T} \|L\|^{n}_\infty \Big(\int_0^t v(s) \,\rd s\Big)^n}{n!} \|x-y\|_\infty\,.
    \end{split}
  \end{equation*}
  Finally, taking into account $\|v\|_1 = \|u\|_1$ we obtain
  \begin{equation*}
    \begin{split} 
      \|F^{n}_u(x) - F^{n}_u(y) \big\|_\infty
      \leq \frac{M^{n} \re^{\mu n T} \|L\|^{n}_\infty \|u\|^n_1}{n!} \|x-y\|_\infty
    \end{split}
  \end{equation*}
  and hence for bounded $B \subset L^1\big([0,T],\RR^m\big)$ we can choose $n \in \NN$ such that
  the estimate $\frac{M^{n} \re^{\mu n T} \|L\|^{n}_\infty \|u\|^n_1}{n!} < 1$ holds for all $u \in B$.
\end{proof}

\begin{remark}
  \label{rem:app-contraction}
  While the above estimation certainly applies to the case $p > 1$ as well, the $\omega$-norm approach
  of the previous subsection fails for $p = 1$. Nevertheless, we decided to use the $\omega$-norm technique
  for $p > 1$ in order to illustrate both methods. 
\end{remark}

Next we need an appropriate replacement of Lemma \ref{lem1:app-compactness}. We already know (cf.~Appendix
\ref{sec:counter-example}) that Lemma \ref{lem1:app-compactness} cannot be ``repaired'' simply by an enhanced
proof technique because the integral operator $F_x: L^1\big([0,T],\RR^m\big) \to C\big([0,T],X\big)$ is in
general not compact. Therefore, the idea is to prove relative compactness of the reachable set $\cR_{\leq T}^{1,r}(\xi_0)$
directly via Corollary \ref{cor:main_3a}. To this end, we define an ``image map''
\begin{equation*}
  \operatorname{Im}: C\big([0,T],X\big) \to \cC(X)
\end{equation*}
via
\begin{equation*}
  \operatorname{Im}(x) := x\big([0,T]\big)
\end{equation*}
where $\cC(X)$ denotes the set of all compact subsets of $X$ equipped with the Hausdorff metric
\begin{equation*}
  d_H(K,K') := \max\Big\{\max_{\xi \in K}\min_{\xi' \in K'}\|\xi-\xi'\|\,, \max_{\xi' \in K'}\min_{\xi \in K}\|\xi-\xi'\|\Big\} \,.
\end{equation*}
As $X$ is complete $\big(\cC(X), d_H\big)$ is complete as well, cf.~\cite[§33.IV, Thm.]{Kuratowski66}.
Moreover, one has the obvious estimate
\begin{equation}
  \label{eq:Lipschitz_G}
  d_H\big(\operatorname{Im}(x),\operatorname{Im}(y)\big) \leq  \|x-y\|_\infty
\end{equation}
for all $x,y\in C([0,T],X)$, i.e.~$\operatorname{Im}$ is Lipschitz continuous with Lipschitz rate $L = 1$.

\begin{proposition}
  \label{prop:compactness_p=1}
  Let $f_i: \RR_0^+ \times X \to X$, $i = 1,\dots, m$ be continuous and let
  $F: C\big([0,T],X\big) \times L^1\big([0,T],\RR\big) \to C\big([0,T],X\big)$ denote the integral operator
  given by \eqref{eq:integral-operator}. 
  Moreover, let $\operatorname{Im}$ be defined as above and set $Z := \big(\cC(X), d_H\big)$.
  Then for $S \subset C\big([0,T],X\big)$ and $B \subset L^1\big([0,T],\RR\big)$ one has the implication:
  \begin{equation*}
    \begin{split}
       \text{$\operatorname{Im}(S)$ relatively comp}&\text{act in $Z$ and $B$ bounded}\\
      &\Downarrow\\
      \text{$\operatorname{Im}\big(F(S \times B)\big)$ is }&\text{relatively compact in $Z$}
    \end{split}
  \end{equation*}
\end{proposition}

\noindent
To prove Proposition \ref{prop:compactness_p=1} the following auxiliary results are useful.

\begin{lemma}[Relative Compactness Criterion\footnote{This result should be well known but it was difficult to locate
    a suitable reference. One implication (``only-if''-part) can be found in \cite[§21.VIII, Thm.~2]{Kuratowski66}.}]
  \label{lem1:compactness_p=1}
  Let a subset $\cK \subset Z$ be given. Then $\cK$ is relatively compact if and only if $\bigcup \cK \subset X$ is
  relatively compact.
\end{lemma}

\begin{proof}
  ``$\Longrightarrow$'': Let $\varepsilon > 0$ and assume that $\cK \subset Z$ is relatively compact with respect to $d_H$.
  Hence there exists a finite $\frac{\varepsilon}{2}$-net $K_1, \dots, K_N$ of $\cK$. Moreover, since all $K_i$,
  $i = 1, \dots, N$ are compact their union is compact as well and hence we can choose a finite $\frac{\varepsilon}{2}$-net
  $\xi_1, \dots, \xi_M$ of it. Now we claim that $\xi_1, \dots, \xi_M$ is an $\varepsilon$-net of $\bigcup\cK$:
  Given $\xi\in\bigcup\cK$. One finds $K\in\cK$ such that $\xi\in K$. For this $K$, there exists $K_i$ such that
  $d_H(K,K_i)<\frac{\varepsilon}{2}$. This implies that one has $\eta\in K_i$ with $\|\xi-\eta\|<\frac{\varepsilon}{2}$.
  Finally, one can choose $\xi_j$ such that $\|\eta-\xi_j\|<\frac{\varepsilon}{2}$. This yields the desired estimate
  \begin{equation*}
    \|\xi - \xi_j\| \leq \|\xi - \eta\| + \|\eta - \xi_j\| < \frac{\varepsilon}{2} + \frac{\varepsilon}{2}
    = \varepsilon\,.
  \end{equation*}

  \noindent
  ``$\Longleftarrow$'': Again, let $\varepsilon > 0$ and assume that $\bigcup \cK \subset X$ is relatively compact.
  Therefore we can choose a finite $\varepsilon$-net $\xi_1, \dots, \xi_M$ of $\bigcup \cK $. Set $\cK_{0}$ be the power set
  of $K_0 := \{\xi_1, \dots, \xi_M\}$. Obviously, $\cK_0$ is a finite collection of compact subsets of $X$. We claim that
  $\cK_{0}$ is the desired $\varepsilon$-net of $\cK$. To see this let $K \in \cK$ be arbitrary and define
  \begin{equation*}
    K' := \{\xi \in K_0  :  K_\varepsilon(\xi) \cap K \neq \emptyset\} \in \cK_{0}\,.
  \end{equation*}
  Now it is a standard exercise (left to the reader) to show $d_H(K,K') \leq \varepsilon$ and we are done.
\end{proof}



\begin{proof}[Proof of Proposition \ref{prop:compactness_p=1}]
  Let $B:=K_1(0)\subset L^1([0,T],\mathbb R)$ and $S \subset  C\big([0,T],X\big)$ such that
  \begin{equation*}
    \operatorname{Im}(S) = \{x([0,T])  :  x \in S\} \subset Z
  \end{equation*}
  is relatively compact. Then, according to Lemma \ref{lem1:compactness_p=1},
  \begin{equation*}
    \operatorname{Ev(S)} = \{x(t)  :  x \in S\,, t \in [0,T]\}
  \end{equation*}
  is relatively compact in $X$ and thus continuity of $f$ implies that the closure of
  $f\big([0,T] \times \operatorname{Ev}(S)\big)$ coincides with
  $ f\big([0,T] \times \overline{\operatorname{Ev}(S)}\big)$ and is therefore compact in $X$.
  Thus we can apply Lemma \ref{lem3:compactness_p=1} to the continuous function $\Gamma:[0,T]\times X\to X$,
  $(t,\xi)\mapsto e^{At}\xi$ and the compact set $K := f\big([0,T] \times \overline{\operatorname{Ev}(S)}\big)$.
  This yields an approximation $\Gamma_{\varepsilon/2}$ such that (i) holds for $\frac{\varepsilon}{2}$. Thus
  for arbitrary $x \in S$ and $u \in B$ it follows 
  \begin{equation}
    \label{eq:epsilon-approximation}
    \begin{split}
      F(x,u)(t) & = \re^{tA}\xi_0 + \int_0^t u(s)\re^{(t-s)A}f\big(s,x(s)\big)\rd s\\
      & = \re^{tA}\xi_0 + \int_0^t u(s)\Gamma\big(t-s,f(s,x(s))\big)\rd s\\
      & = \re^{tA}\xi_0 + \int_0^t u(s)\Gamma_{\varepsilon/2}\big(t-s,f(s,x(s))\big)\rd s\\
      &\hphantom{= \re^{tA}\xi_0} + \int_0^t u(s) \Big(\Gamma\big(t-s,f(s,x(s))\big) - \Gamma_{\varepsilon/2}\big(t-s,f(s,x(s))\big)\Big)\rd s\,.
    \end{split}
  \end{equation}
  Note that integrability of $s \mapsto \Gamma_{\varepsilon/2}\big(t-s,f(s,x(s))\big)$ follows from its measurability
  and essential boundedness. Since continuity of $t \mapsto \re^{tA}\xi_0$ implies that $\{\re^{tA}\xi_0  :  t \in [0, T ]\}$
  is compact it suffices to focus on the second and third term of \eqref{eq:epsilon-approximation}. For the latter, we get
  \begin{equation*}
    \begin{split}
      & \bigg\|\int_0^t u(s) \Gamma\big(t-s,f(s,x(s))\big) \,\rd s
      - \int_0^t u(s) \Gamma_{\varepsilon/2}\big(t-s,f(s,x(s))\big) \,\rd s \bigg\|\\
      & \qquad
      \leq \int_0^t |u(s)| \Big\|\Gamma\big(t-s,f(s,x(s))\big) - \Gamma_{\varepsilon/2}\big(t-s,f(s,x(s))\big)\Big\| \,\rd s\\
      & \qquad < \frac{\varepsilon}{2} \int_0^t |u(s)|  \,\rd s \leq  \frac{\varepsilon}{2}\,.
    \end{split}
  \end{equation*}
  Hence we conclude
  \begin{equation}
    \label{eq:infinity-norm-estimate}
    \bigg\|\int_0^{(\cdot)} u(s) \Gamma\big(\,\cdot\, -s,f(s,x(s))\big) \,\rd s
    -  \int_0^{(\cdot)} u(s) \Gamma_\varepsilon\big(\,\cdot\, -s,f(s,x(s))\big) \,\rd s \bigg\|_\infty \leq \frac{\varepsilon}{2}\,.
  \end{equation}
  For the second term we get the following representation
  \begin{equation*}
    \begin{split}
       \int_0^t u(s)\Gamma_{\varepsilon/2}\big(t-s,f(s,x(s))\big)\rd s 
      &  = \sum_{i=1}^N\sum_{j=1}^M \int_0^t u(s) \chi_{\Delta_i}(t-s)\chi_{S_j}\big(f(s,x(s))\big) \xi_{ij} \,\rd s \\
      &  = \sum_{i=1}^N\sum_{j=1}^M \int\limits_{I_{ij}(t)}u(s) \,\rd s \, \xi_{ij}
    \end{split}
  \end{equation*}
  with $I_{ij}(t) := [0,t] \cap (t - \Delta_{i}) \cap (f \circ \big(\id \times \, x(\cdot))\big)^{-1}(S_j) $. This shows
  \begin{equation*}
    \begin{split}
       \int_0^t u(s)\Gamma_{\varepsilon/2}\big(t-s,f(s,x(s))\big)\rd s 
      &= \sum_{i=1}^N\sum_{j=1}^M \int\limits_{I_{ij}(t)}u(s) \,\rd s \, \xi_{ij} \\
      & \subset \bigg\{\sum_{i=1}^N\sum_{j=1}^M \lambda_{ij}\xi_{ij} \;\bigg|\;\lambda_{ij} \in [-1,1]\bigg\} =: C \subset X
    \end{split}
  \end{equation*}
  and hence the second term is contained in the compact convex subset $C$. Thus, one can choose a finite
  $\frac{\varepsilon}{2}$-net $y_1, \dots, y_L$ for $C$ which yields---due to \eqref{eq:infinity-norm-estimate}---a finite
  $\varepsilon$-net of
  \begin{equation*}
    \bigg\{ \int_0^t u(s)e^{(t-s)A}f\big(s,x(s)\big)\rd s \;:\; t \in [0,T]\,, x \in S\,,u\in B\bigg\}\,.
  \end{equation*}
  Since the sum of relatively compact subsets is again relatively compact
  we conclude that $\operatorname{Ev}\big(F(S \times B)\big)$ is relatively compact. A further application of Lemma
  \ref{lem1:compactness_p=1} proves relative compactness of $\operatorname{Im}\big(F(S \times B)\big) \subset Z$.
\end{proof}

\noindent
After these preliminaries we are well-prepared for proving the main result of this subsection.

\begin{proof}[Proof of Theorem \ref{thm:app_p=1}]
  As in the proof of Corollary \ref{cor:app_p>1}, we have to show that for fixed $T \geq 0$ and $n \in \NN$,
  the reachable set up to time $T$ under bounded $L^1$-controls, i.e.
  \begin{equation*}
    R_{\leq T}^{1,n}(\xi_0) := \{x(t,\xi_0,u)  :  t \in [0,T]\,, \|u\|_1 \leq n \}
  \end{equation*}
  is relatively compact. The rest follows again immediately from Baire's category theorem and the identity
  \begin{equation*}
    R^1(\xi_0) = \bigcup_{T \in \NN} \bigcup_{n \in \NN}  R_{\leq T}^{1,n}(\xi_0)\,.
  \end{equation*}
  Hence, let $T \geq 0$ be fixed and, for simplicity, let $m=1$. Moreover, by Lemma \ref{lem:cutoff}, we
  can assume without loss of generality that $f := f_1$ is globally Lipschitz in $\xi$ on bounded subsets of
  $L^1\big([0,T],\RR\big)$.

  \medskip
  \noindent
  Next, let $B := K_n(0) \subset L^1\big([0,T],\RR\big)$ denote the closed ball around the origin with radius $n \in \NN$.
  By Lemma \ref{lem:equivalent-metric} and \ref{lem5:app-contraction}, we can choose a metric $d'$ on $C\big([0,T],X\big)$
  which is strongly equivalent to the metric induced by the maximum-norm $\|\cdot\|_\infty$ such that the restricted integral
  operator $F: C\big([0,T],X\big) \times B \to C\big([0,T],X\big)$ constitutes a uniform contraction in $x$. Then
  Proposition \ref{prop:compactness_p=1} and Eq.~\eqref{eq:Lipschitz_G} allow us to apply Corollary \ref{cor:main_3a}
  to the restriction of $F$. Finally, Lemma \ref{lem1:compactness_p=1} yields the desired relative compactness of
  $\operatorname{Ev}\big(\Phi(B)\big) = \cR_{\leq T}^{1,n}(\xi_0)$. This concludes the proof.
\end{proof}


\appendix

\section{Counter-Example} 
\label{sec:counter-example}

Consider the scalar system
\begin{equation*}
  \dot{x}(t) = u(t)\,, \quad x(0) = 0 \quad \text{and} \quad u \in L^1([0,1],\RR)\,.
\end{equation*}
and let $\Phi: L^1([0,1],\RR) \to C\big([0,1],\RR\big)$ denote the corresponding fixed point map,
cf.~Theorem \ref{thm:app_p>1}. Then the image of the closed unit ball of $L^1([0,1],\RR)$ under $\Phi$
obviously contains the sequence $(x_n)_{n \in \NN}$ of continuous functions given by
\begin{equation*}
  x_n(t) = \begin{cases}
    nt & \text{for $t \in [0,\tfrac{1}{n}]$,}\\[2mm]
    1 & \text{for $t \in [\tfrac{1}{n},1]$,}
  \end{cases}
\end{equation*}
as $x_n = \Phi(u_n)$ with $u_n(t) := n$ for $t\in[0,\frac{1}{n}]$ and $u_n(t) := 0$ else. Evidently, the
sequence $(x_n)_{n \in \NN}$ is not equicontinuous (at $t=0$) so the image $\Phi\big(K_1(0)\big)$ is
not relatively compact due to Arzel\`{a}--Ascoli (alternatively, one may argue that $(x_n)_{n \in \NN}$
does not possess a convergent subsequence).


\section{A Technical Lemma}
\label{sec:proof-lem3:compactness_p=1}

\begin{lemma}
  \label{lem3:compactness_p=1}
  Let $T\geq 0$, $K \subset X$ be compact, and $\Gamma:[0,T]\times X\to X$ be continuous. Then for every $\varepsilon > 0$
  there exists $\Gamma_\varepsilon: [0,T] \times K\to X$ such that the following conditions are satisfied:\smallskip
  \begin{enumerate}
  \item[(i)]
    For all $\xi \in K$ and $t\in[0,T]$ one has $\|\Gamma(t,\xi) - \Gamma_\varepsilon(t,\xi)\| < \varepsilon$.\smallskip
  \item[(ii)]
    The image of $\Gamma_\varepsilon$ is finite. In particular, there exist finitely many disjoint Borel
    sets $\Delta_i \subset [0,T]$, $i = 1, \dots, N$ and $S_j \subset K$, $j = 1, \dots, M$
    as well as $\xi_{ij} \in X$ such that
    \begin{equation*}
      \label{eq:pc-approximation}
      \Gamma_\varepsilon(t,\xi) = \sum_{i=1}^N\sum_{j=1}^M \chi_{\Delta_i}(t)\chi_{S_j}(\xi)\xi_{ij}
    \end{equation*}  
    for all $\xi \in K$ and $t\in[0,T]$.
  \end{enumerate}
\end{lemma}

\begin{proof}
  Let $K \subset X$ be compact and $\varepsilon > 0$. For convenience, choose $d_\infty$ as metric
  on $[0,T] \times K$, cf.~\eqref{eq:product-metric_inf}. Due to the compactness of $[0,T] \times K$
  the restriction $\Gamma\big|_{[0,T] \times K}$ is uniformly continuous 
  meaning there exists $\delta > 0$ such that
  \begin{equation*}
    \big\|\Gamma(t',\xi') - \Gamma(t,\xi)\big\| \leq \varepsilon
  \end{equation*}
  for $d_\infty\big((t',\xi'),(t,\xi)\big)< \delta$, i.e.~for $|t'-t| < \delta$  and $\|\xi'-\xi\| < \delta$.
  Again the compactness of $K$ implies that there exists a finite $\delta$-net $\eta_1, \dots, \eta_M$ of $K$.
  Then, choosing $N \in \NN$ such that $T/N < \delta$ allows to define $\Gamma_{\varepsilon}$ as follows:
  \begin{equation*}
    \Gamma_{\varepsilon}(t,\xi) :=
    \begin{cases}
      \Gamma(\tfrac{T}{N}, \eta_1) & \text{for $t \in \big[0,\tfrac{T}{N}\big]$ and $\xi \in K \cap B_\delta(\eta_1)$,}\\[2mm]
      \Gamma(\tfrac{2T}{N}, \eta_1) & \text{for $t \in \big(\tfrac{T}{N},\tfrac{2T}{N}\big]$ and $\xi \in K \cap B_\delta(\eta_1)$,}\\
      \vdots & \vdots\\
      \Gamma(T, \eta_1) & \text{for $t \in \big(\tfrac{(N-1)T}{N}, T\big]$ and $\xi \in K \cap B_\delta(\eta_1)$,}\\[2mm]
      \Gamma(\tfrac{T}{N}, \eta_2) & \text{for $t \in \big[0,\tfrac{T}{N}\big]$ and
        $\xi \in K \cap \Big(B_\delta(\eta_2) \setminus B_\delta(\eta_1)\Big)$,}\\
      \vdots & \vdots\\
      \Gamma(T, \eta_M) & \text{for $t \in \big(\tfrac{(N-1)T}{N},T\big]$ and
        $\xi \in K \cap \Big(B_\delta(\eta_M) \setminus \bigcup_{k=1}^{M-1}B_\delta(\eta_k)\Big)$.}
    \end{cases}
  \end{equation*}
  The straightforward proof that $\Gamma_{\varepsilon}$ satisfies (i) is left to the reader. Finally, setting
  $\Delta_1 := \big[0,\frac{T}{N}\big]$, $\Delta_i := \big(\frac{(i-1)T}{N},\frac{iT}{N}\big]$ for $i = 2, \dots, N$,
  $S_j := K \cap \Big(B_\delta(\eta_j) \setminus \bigcup_{k=1}^{j-1}B_\delta(\eta_k)\Big)$ for $j = 1, \dots, M$
  and $\xi_{ij} := \Gamma(\frac{iT}{N},\eta_j)$ yields the desired representation \eqref{eq:pc-approximation}
  of $\Gamma_{\varepsilon}$.
\end{proof}


\section*{Acknowledgments}
I would like to thank my colleague and co-author of several publications, Frederik vom Ende, for his careful proofreading
and valuable comments. His feedback substantially enhanced the presentation of the subject matter. Moreover,
  I also appreciate the reviewers' careful proofreading and their suggestions in particular concerning missing literature.
  Finally, I would like to thank Jochen Gl\"uck for pointing out a beautiful alternative proof of Theorem \ref{thm:main_1}
  (see Remark \ref{rem:alternative_proof}).

 \newpage


\end{document}